\documentclass[10pt,a4paper]{amsart}
\usepackage{fullpage}
\usepackage{amsmath}
\usepackage{amsthm}
\usepackage{amssymb}
\usepackage{mathtools}
\usepackage[only,mapsfrom]{stmaryrd}
\usepackage{graphicx}
\usepackage{cancel} 
\usepackage{color} 
\usepackage[all]{xy} 
\usepackage{tikz}
\usetikzlibrary{arrows,calc}
\usetikzlibrary{cd}
\usepackage{nicefrac}
\usepackage{paralist}
\usepackage{enumitem}
\usepackage[colorlinks,citecolor=blue,linkcolor=blue,urlcolor=blue,filecolor=blue,breaklinks]{hyperref}

\numberwithin{equation}{section}
\numberwithin{figure}{section}

\newenvironment{itemizeb}%
{\begin{compactitem}

}%
{\end{compactitem}}

\newtheorem{thm}{Theorem}[section]
\newtheorem{athm}{Theorem}

\newtheorem{lem}[thm]{Lemma}
\newtheorem{prop}[thm]{Proposition}
\newtheorem{cor}[thm]{Corollary}

\newtheorem{ques}[thm]{Question}

\newtheorem*{thm*}{Theorem}
\newtheorem*{conj*}{Conjecture}
\newtheorem*{cor*}{Corollary}
\newtheorem*{ques*}{Question}

\theoremstyle{definition}
\newtheorem{rem}[thm]{Remark}
\newtheorem{defn}[thm]{Definition}
\newtheorem{ex}[thm]{Example}
\newtheorem{notation}[thm]{Notation}
\newtheorem*{rem*}{Remark}

\newcommand{\cC}{\mathcal{C}}

\newcommand{\cF}{\mathcal{F}}

\newcommand{\cP}{\mathcal{P}}

\newcommand{\cR}{\mathcal{R}}

\newcommand{\bD}{\mathbb{D}}

\newcommand{\bN}{\mathbb{N}}

\newcommand{\bQ}{\mathbb{Q}}
\newcommand{\bR}{\mathbb{R}}
\newcommand{\bS}{\mathbb{S}}

\newcommand{\bZ}{\mathbb{Z}}

 \newcommand{\BN}{{\mathbb {N}}}
 
 \newcommand{\BR}{{\mathbb {R}}}

 \newcommand{\BZ}{{\mathbb {Z}}}
 
\newcommand{\CC}{{\mathcal {C}}}

 \newcommand{\Aut}{{\mathrm{Aut}}}

\DeclareMathOperator{\Map}{Map}
\DeclareMathOperator{\Homeo}{Homeo}
\DeclareMathOperator{\PMap}{PMap}

\newcommand{\fc}{{\mathfrak{c}}}
\newcommand{\PMapcbar}[1]{\ensuremath{\overline{\PMap_c(#1)}}}
\newcommand{\Ends}{{\mathrm{Ends}}}
\newcommand{\longhookrightarrow}{\ensuremath{\lhook\joinrel\longrightarrow}}
\newcommand{\incl}[3][right]%
{%
\draw[<-,>=#1 hook] #2 to ($ #2!0.5!#3 $);
\draw[->] ($ #2!0.5!#3 $) to #3;%
}

\title{Big mapping class groups with uncountable integral homology}

\author{Martin Palmer}
\address{Institutul de Matematică Simion Stoilow al Academiei Române, 21 Calea Griviței, 010702 Bucharest, Romania}
\email{mpanghel@imar.ro}

\author{Xiaolei Wu}
\address{Shanghai Center for Mathematical Sciences, Jiangwan Campus, Fudan University, No.2005 Songhu Road, Shanghai, 200438, P.R. China}
\email{xiaoleiwu@fudan.edu.cn}

\date{12 October 2023}

\begin{document}

\begin{abstract}
We prove that, for any infinite-type surface $S$, the integral homology of the closure of the compactly-supported mapping class group $\PMapcbar{S}$ and of the Torelli group $\mathcal{T}(S)$ is uncountable in every positive degree. By our results in \cite{PalmerWu1} and other known computations, such a statement cannot be true for the full mapping class group $\Map(S)$ for all infinite-type surfaces $S$. However, we are still able to prove that the integral homology of $\Map(S)$ is uncountable in all positive degrees for a large class of infinite-type surfaces $S$. The key property of this class of surfaces is, roughly, that the space of ends of the surface $S$ contains a limit point of topologically distinguished points. Our result includes in particular all finite-genus surfaces having countable end spaces with a unique point of maximal Cantor-Bendixson rank $\alpha$, where $\alpha$ is a successor ordinal. We also observe an order-$10$ element in the first homology of the pure mapping class group of any surface of genus $2$, answering a recent question of G.~Domat.
\end{abstract}
\maketitle

\section*{Introduction}

There has been a recent wave of interest in \emph{big mapping class groups} (mapping class groups of infinite-type surfaces); see \cite{AV20} for a survey. In \cite{PalmerWu1}, the authors recently computed the homology of a large family of big mapping class groups, namely the families of (\emph{$1$-holed} or \emph{punctured}) \emph{binary tree surfaces} (see the introduction of \cite{PalmerWu1} for this terminology). Precisely, the mapping class group of every $1$-holed binary tree surface is \emph{acyclic} and the homology of the mapping class group of every punctured binary tree surface is periodic with $\bZ$ in every even degree and zero in every odd degree. One instance of this result says that the mapping class group $\Map(\bD^2 \smallsetminus \cC)$ is acyclic and that $H_i(\Map(\bR^2 \smallsetminus \cC))$ is $\bZ$ for $i$ even and zero for $i$ odd, where $\cC$ is a Cantor set embedded in the interior of the disc. In particular, in all of these examples, the homology groups $H_i(\Map(S))$ are \emph{finitely generated} for each $i$. Some earlier results on the homology of big mapping class groups -- in degrees $1$ and $2$ -- include: $H_1(\Map(S \smallsetminus \cC)) \cong H_1(\Map(S))$ if $\cC$ is a Cantor set embedded in the interior of a finite-type surface $S$ \cite{CalegariChen2022} (see also \cite{Vlamis21} for three special cases of this) and $H_2(\Map(\bS^2 \smallsetminus \cC)) \cong \bZ/2$ \cite{CC21}.

In this paper we prove a contrasting result: for many infinite-type surfaces $S$, the group $H_i(\Map(S))$ is \emph{uncountable} for all $i>0$. In addition, we prove -- for \emph{all} infinite-type surfaces $S$ -- that $H_i(\PMapcbar{S})$ and $H_i(\mathcal{T}(S))$ are uncountable for all $i>0$, where $\PMapcbar{S}$ and $\mathcal{T}(S)$ denote, respectively, the closure of the compactly-supported mapping class group and the Torelli group of $S$.

Our proofs are built on ideas from \cite{APV20,Domat2022,MalesteinTao21}. In \cite{APV20}, Aramayona, Patel and Vlamis determined $H^1(\PMap(S))$ for any infinite-type surface $S$ of genus at least $2$; in particular, they showed that it is countable. (This was extended to genus $1$ in \cite{DomatPlummer2020}, where it was also shown that $H^1(\PMap(S))$ is uncountable when $S$ has genus $0$.) Along the way they proved that, when $S$ has infinitely many non-planar ends, its pure mapping class group $\PMap(S)$ admits a split-surjection onto the \emph{Baer-Specker group} $\BZ^\BN$. Later, Domat proved that big pure mapping class groups $\PMap(S)$ are never perfect \cite{Domat2022}. Morover, he showed that $H_1(\PMap(S))$ is uncountable for many infinite-type surfaces $S$ and that $H_1(\mathcal{T}(S))$ and $H_1(\PMapcbar{S})$ are uncountable for all infinite-type surfaces $S$. Malestein and Tao \cite{MalesteinTao21} were able to push the results of Domat further and prove that the first homology of the full mapping class group $H_1(\Map(S))$ is uncountable for a certain class of surfaces $S$, including $S = \BR^2 \smallsetminus \BZ$.

\subsection*{Uncountable homology.}

Given a surface $S$, recall that its pure mapping class group $\PMap(S)$ is the subgroup of its mapping class group $\Map(S) = \pi_0(\Homeo(S))$ consisting of all those mapping classes that fix the ends of $S$ pointwise. Its Torelli group $\mathcal{T}(S)$ is the kernel of the natural homomorphism $\Map(S)\to \Aut(H_1(S))$. Recall also that $\PMap_c(S)$ denotes the subgroup of $\Map(S)$ of mapping classes that admit representative homeomorphisms with compact support, and $\PMapcbar{S}$ denotes its closure in $\Map(S)$ in the quotient topology induced by the compact-open topology on $\Homeo(S)$. We note that in general we have inclusions $\mathcal{T}(S) \subseteq \PMapcbar{S} \subseteq \PMap(S) \subseteq \Map(S)$. (The only non-obvious inclusion is the first one: it is explained during the proof of Theorem \ref{thm:torelli} below.) Our first result concerns the first two groups of this nested sequence and holds for \emph{all} infinite-type surfaces $S$.

\begin{athm}[Corollary \ref{cor:closure-of-compactly-supported-mcg} and Theorem \ref{thm:torelli}]
\label{athm:PMap}
Let $S$ be any infinite-type surface. Then the integral homology groups
\[
H_i(\PMapcbar{S}) \quad\text{and}\quad H_i(\mathcal{T}(S))
\]
are uncountable for every $i \geq 1$. Moreover, they each contain $\bigoplus_{\fc} \bZ$ in every degree, where $\fc$ denotes the cardinality of the continuum.
\end{athm}

\begin{rem}
One might hope that our methods could be used to prove that the homology of the pure mapping class group $H_i(\PMap(S))$ is also uncountable for every $i \geq 1$ and for any infinite-type surface $S$. However, the methods of the present paper can only prove this result in the case when $S$ has at most one or infinitely many non-planar ends; see Remark \ref{rem:hmg-PMap-uncountable} for more information. When $S$ has $n$ non-planar ends for $1 < n < \infty$, one can in fact prove that the (uncountably many) elements constructed in Domat's paper \cite[Theorem 6.1]{Domat2022} all vanish in $H_1(\PMap(S))$; see Remark \ref{rem:gap-pmap-hmg} for more information.
\end{rem}

In order to state our result for the full mapping class groups $\Map(S)$, we first recall some background about ends of surfaces; more details are given in \S\ref{s:surfaces} and \S\ref{s:top-distinguished}. Every surface $S$ has a space of ends $E$, which is a compact, separable, totally disconnected topological space. The key hypothesis in our main theorem is a condition on the structure of the space $E$.

\begin{defn}
\label{defn:topological-type-partition}
For points $x,y \in E$, we write $x \sim y$ and say that $x$ is \emph{similar} to $y$ if and only if there are open neighbourhoods $U,V$ of $x,y$ respectively such that $(U,x)$ and $(V,y)$ are homeomorphic as based spaces. A point $x \in E$ is \emph{topologically distinguished} if it is not equivalent to any other point of $E$ under this equivalence relation.
\end{defn}

\begin{defn}
\label{defn:Upsilon}
For a topological space $E$, write $\Upsilon^+(E) = E\omega + 1$, where $E\omega$ means a countably infinite disjoint union of copies of $E$ and $X+1$ means the one-point compactification of $X$.
\end{defn}

\begin{athm}
\label{athm:uncountable}
Let $S$ be a connected, finite-genus surface with finitely many boundary components, whose space of ends $E$ is of the form $E = E_1 \sqcup \Upsilon^+(E_2)$, where $E_2$ has a topologically distinguished point $x$ and no point of $E_1$ is similar to $x$.  Then the integral homology group
\[
H_i(\Map(S))
\]
is uncountable for every $i\geq 1$. In fact, there is an injective homomorphism of graded abelian groups
\[
\Lambda^*\Bigl(\bigoplus_{\fc} \bZ\Bigr) \longrightarrow H_*(\Map(S)),
\]
where $\Lambda^*$ denotes the exterior algebra on an abelian group.
\end{athm}

\begin{rem}
In the course of the proof of Theorem \ref{athm:uncountable}, we also prove the same statement with $S$ replaced by the Loch Ness monster surface $L$, see Proposition \ref{prop:uncountability-Loch-Ness}.
\end{rem}

\begin{rem}
All \emph{countable} end spaces of surfaces (equivalently: countable compact Hausdorff spaces) are of the form $E = \omega^\alpha.n+1$ for a countable ordinal $\alpha$ and a positive integer $n$ \cite{MazurkiewiczSierpinski1920}. Hence a surface $S$ of finite genus with this end space satisfies the assumption of Theorem \ref{athm:uncountable} whenever $n=1$ and $\alpha$ is a successor ordinal.
\end{rem}

Thus for a large class of infinite-type surfaces $S$ with countably many ends we know that $\Map(S)$ has uncountable integral homology in all positive degrees. This suggests the following question.

\begin{ques}
\label{ques-countable}
Let $S$ be an infinite-type surface with countably many ends. Is the homology of $\Map(S)$  uncountable in all positive degrees?
\end{ques}

\begin{rem}
Without the hypothesis on the structure of the space of ends $E$ of $S$, the conclusion of Theorem \ref{athm:uncountable} is false. For example, as mentioned above, we prove in \cite{PalmerWu1} that
\[
H_i(\Map(\bR^2 \smallsetminus \cC)) \cong \begin{cases}
\bZ & i \text{ even} \\
0 & i \text{ odd}.
\end{cases}
\]
\end{rem}

\begin{rem}
The hypotheses of this paper and the hypotheses of \cite{PalmerWu1} are in some sense opposite, with opposite conclusions. In \cite{PalmerWu1} we consider $1$-holed binary tree surfaces, whose end spaces are \emph{Cantor compactifications} $(E\omega)^\cC$ (see \cite[\S 1.2]{PalmerWu1} for the definition), which are highly self-similar (in particular $(E\omega)^\cC \cong \cC$ if $E = \varnothing$ or $E=\cC$, which is homogeneous), and we prove that $H_i(\Map(S)) = 0$ for all $i>0$. On the other hand, in this paper we consider surfaces $S$ whose end spaces $E$ satisfy the ``homogeneity breaking'' hypothesis of Theorem \ref{athm:uncountable} (roughly: $E$ has a limit point of topologically distinguished points), and conclude that $H_i(\Map(S))$ is uncountable for all $i>0$.
\end{rem}

\subsection*{Non-trivial torsion.}

So far, the elements that we have constructed in the homology of big mapping class groups all have infinite order. It would be interesting also to find some torsion elements. In fact, the following question was asked by Domat in \cite[Question 11.3]{Domat2022}.

\begin{ques}
\label{ques-tor-hml}
Let $S$ be an infinite-type surface. Are there torsion elements in $H_1(\PMapcbar{S})$?
\end{ques}

Recall that $\PMap_c(S)$ denotes the subgroup of $\Map(S)$ of mapping classes that admit representative homeomorphisms with compact support, and $\PMapcbar{S}$ denotes its closure in $\Map(S)$ in the quotient topology induced by the compact-open topology on $\Homeo(S)$. Also recall that $\PMapcbar{S} \subseteq \PMap(S)$ coincides with $\PMap(S)$ if and only if $S$ has at most one non-planar end \cite[Theorem 4]{PriyamVlamis18}. Our third result answers Domat's question in the positive.

\begin{athm}
\label{thm-tor10-hmlg}
Let $S$ be an infinite-type surface of genus $2$ and with finitely many (possibly zero) boundary components. Then the homology groups $H_1(\PMap(S)) = H_1(\PMapcbar{S})$ and $H_1(\Map(S))$ both contain an order-$10$ element. Moreover, the cyclic group generated by this element is a direct summand.
\end{athm}

\begin{rem}
By comparing the stable homology of (orientable, finite-type) mapping class groups with rational coefficients \cite{MW07} and with mod-$p$ coefficients \cite{Galatius04}, one sees that there are also many torsion elements in the integral homology of mapping class groups in the stable range. Using this and Lemma \ref{lem:embedding}, one may find many higher-degree torsion elements in the homology of mapping class groups of infinite-type surfaces of finite genus.
\end{rem}

In a sense, our answer to Domat's question is ``cheating'', since we simply show that a certain torsion element in the homology of the mapping class group of a finite-type subsurface of $S$ injects into the homology of the mapping class group of $S$. Together with our uncountability results above (Theorems \ref{athm:PMap} and \ref{athm:uncountable}), this suggests two refinements of Domat's question:

\begin{ques}
\label{ques-torsion-infinite-support}
Let $S$ be an infinite-type surface. Do the homology groups $H_1(\PMapcbar{S})$ or $H_1(\PMap(S))$ contain torsion elements that are not supported on any finite-type subsurface of $S$?
\end{ques}

\begin{ques}
\label{ques-torsion-uncountable}
Let $S$ be an infinite-type surface. Do the homology groups $H_1(\PMapcbar{S})$ or $H_1(\PMap(S))$ contain an uncountable torsion subgroup?
\end{ques}

We note that a positive answer to Question \ref{ques-torsion-uncountable} would imply a positive answer to Question \ref{ques-torsion-infinite-support}, since torsion admitting finite-type support can only account for countably many torsion elements.

\subsection*{Outline.} 

We begin with two sections of background: \S\ref{s:surfaces} on infinite-type surfaces and big mapping class groups and \S\ref{s:top-distinguished} on notions of \emph{topologically distinguished points}. In \S\ref{s:lifting-uncountability} we prove a basic lemma that gives a sufficient criterion for the homology of a group to contain an embedded copy of the exterior algebra on a direct sum of continuum many copies of $\bZ$. We also discuss techniques of \cite{Domat2022} that may be used to construct the inputs for this lemma.

Theorems \ref{athm:PMap} and \ref{athm:uncountable} are then proven in \S\ref{s:Loch-Ness}--\S\ref{s:double-branched-covers}. In \S\ref{s:Loch-Ness} we prove uncountability of the homology of the mapping class group of the Loch Ness monster surface, which is the first step in the proof of Theorem \ref{athm:uncountable}. We then adapt these techniques to prove Theorem~\ref{athm:PMap} on the homology of the closure of the compactly-supported mapping class group and the Torelli group of an arbitrary infinite-type surface $S$. In \S\ref{s:double-branched-covers} we apply the results of \S\ref{s:Loch-Ness}, together with a covering space argument inspired by a technique of Malestein and Tao~\cite{MalesteinTao21}, to complete the proof of Theorem \ref{athm:uncountable}. The covering space argument in this section is the step in which we use in an essential way the hypothesis on the structure of the end space of the surface.

We prove Theorem \ref{thm-tor10-hmlg} on torsion elements in \S\ref{s:torsion}. Finally, in \S\ref{s:open-problems}, we record some related open questions, in particular discussing the \emph{co}homology of mapping class groups in \S\ref{ss:open-cohomology}. Appendix \ref{appendix:abelian-groups} gathers some basic facts about abelian groups that are needed in several of our proofs.

\subsection*{Acknowledgements.}

MP was partially supported by a grant of the Romanian Ministry of Education and Research, CNCS - UEFISCDI, project number PN-III-P4-ID-PCE-2020-2798, within PNCDI III. XW is currently a member of LMNS and supported by a starter grant at Fudan University. He  thanks Guozhen Wang for discussions related to the mod-$p$ homology of the stable mapping class group.

\section{Surfaces, ends and mapping class groups}
\label{s:surfaces}

\subsection{Infinite-type surfaces}\label{s:inf-type}

All surfaces will be assumed to be second countable, connected, orientable and to have compact boundary. If the fundamental group of $S$ is finitely generated, we say that $S$ has \emph{finite type}, otherwise it has \emph{infinite type}. The classification of surfaces of possibly infinite type was proven by von Ker\'ekj\'art\'o \cite{vKer23} and Richards \cite{Ri63}. Recall that an \emph{end} of a surface $S$ is an element of the set
\begin{equation}
\label{eq:inverse-limit}
\Ends(S) = \varprojlim \pi_0( S\setminus K),
\end{equation}
where the inverse limit is taken over all compact subsets $K \subset S$. The \emph{Freudenthal compactification} of $S$ is the union
\[
\overline{S} = S \sqcup \Ends(S)
\]
equipped with the topology generated by $U \sqcup \{ e \in \Ends(S) \mid e<U \}$ for all open subsets $U \subseteq S$. Here $e<U$ means that there is a compact subset $K \subset S$ such that $U$ contains the component of $S \setminus K$ hit by $e$ under the canonical map $\Ends(S) \to \pi_0(S \setminus K)$. The induced subspace topology on $\Ends(S)$ coincides with the limit topology induced from the discrete topology on each term in the inverse system. With this topology, $\Ends(S)$ is homeomorphic to a closed subset of the Cantor set. We call an end $e \in \Ends(S)$ \emph{planar} if it has a neighbourhood (in the topology of $\overline{S}$) that embeds into the plane, otherwise we call it \emph{non-planar}. The (closed) subspace of non-planar ends is denoted by $\Ends_{np}(S) \subseteq \Ends(S)$.

\begin{thm}[{\cite[Theorems 1 and 2]{Ri63}}]\label{thm:clas-inf-sur}
Let $S_1,S_2$ be two surfaces of genus $g_1,g_2 \in \bN \cup \{\infty\}$ and with $b_1,b_2 \in \bN$ boundary components. Then $S_1 \cong S_2$ if and only if $g_1 = g_2$, $b_1 = b_2$ and there is a homeomorphism of pairs of spaces
\[
(\Ends(S_1), \Ends_{np}(S_1)) \;\cong\; (\Ends(S_2), \Ends_{np}(S_2)).
\]
Conversely, given $g \in \bN \cup \{\infty\}$, $b \in \bN$ and a pair $X\subseteq Y$ of closed subsets of the Cantor set, where we require that $g = \infty$ if and only if $X \neq \varnothing$, there exists a surface $S$ of genus $g$ with $b$ boundary components such that $(\Ends(S),\Ends_{np}(S)) \cong (Y,X)$.
\end{thm}

\subsection{Mapping class groups}

For a surface $S$, the \emph{mapping class group} of $S$ is the group of isotopy classes of orientation-preserving diffeomorphisms of $S$ fixing the boundary of $S$ pointwise, i.e.
\[
\Map(S) := \pi_0(\mathrm{Diff}^+(S,\partial S)).
\]

The \emph{pure mapping class group} $\PMap(S)$ of $S$ is the subgroup of $\Map(S)$ consisting of all elements whose induced action on $\Ends(S)$ is the identity. It follows from the construction of \cite[Theorem~2]{Ri63} (or, more precisely, from the \emph{naturality} of this construction) that every homeomorphism of $\Ends(S)$ sending the subspace $\Ends_{np}(S)$ onto itself is induced by some homeomorphism of $S$. This implies that we have the following short exact sequence.

\begin{prop}\label{thm:ses-bmcg}
Let $S$ be any surface. Then there is a short exact sequence of groups
\[
1 \to \PMap(S) \longrightarrow \Map(S) \longrightarrow \mathrm{Homeo}(\Ends(S),\Ends_{np}(S)) \to 1,
\]
where $\mathrm{Homeo}(\Ends(S),\Ends_{np}(S))$ is the group of homeomorphisms of the pair $(\Ends(S),\Ends_{np}(S))$.
\end{prop}

\section{Topologically distinguished points}
\label{s:top-distinguished}

We now recall from the introduction the notion of \emph{topologically distinguished points} (Definition \ref{defn:topological-type-partition}) and compare it to a weaker notion of \emph{globally topologically distinguished points}.

\begin{defn}
\label{defn:topologically-distinguished}
Let $E$ be a topological space. Two points $x,y \in E$ are called \emph{similar} if there are open neighbourhoods $U$ and $V$ of $x$ and $y$ respectively and a homeomorphism $U \cong V$ taking $x$ to $y$. This is an equivalence relation on $E$. A point $x \in E$ is called \emph{topologically distinguished} if its equivalence class under this relation is $\{x\}$, in other words it is similar only to itself.
\end{defn}

\begin{defn}
\label{defn:globally-topologically-distinguished}
Let $E$ be a topological space. Two points $x,y \in E$ are called \emph{globally similar} if there is a homeomorphism $\varphi \in \Homeo(E)$ with $\varphi(x)=y$. This is an equivalence relation on $E$. A point $x \in E$ is called \emph{globally topologically distinguished} if its equivalence class under this relation is $\{x\}$, in other words it is similar only to itself. Equivalently, $x \in E$ is globally topologically distinguished if it is a fixed point of the action of $\Homeo(E)$ on $E$.
\end{defn}

\begin{figure}[tb]
    \centering
    \includegraphics{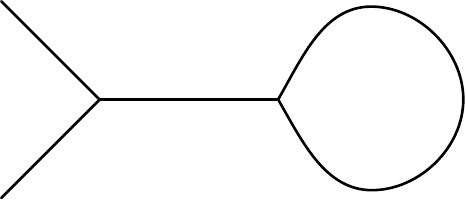}
    \caption{The $3$-valent vertices of this graph are globally topologically distinguished but not topologically distinguished, since they are similar (but not globally similar) to each other.}
    \label{fig:graph-counterexample}
\end{figure}

\begin{rem}
We record two immediate observations:
\begin{itemizeb}
\item If $x$ and $y$ are globally similar then they are similar.
\item If $x$ is topologically distinguished then it is globally topologically distinguished.
\end{itemizeb}
The converses of these two statements are false in general. For example, the two vertices of valence $3$ in the graph pictured in Figure \ref{fig:graph-counterexample} are similar but not globally similar; also, both of them are globally topologically distinguished but not topologically distinguished. However, for zero-dimensional (Hausdorff) spaces the converse does hold:
\end{rem}

\begin{lem}
\label{lem:global-to-local}
Suppose that $E$ is Hausdorff and zero-dimensional, i.e.~it has a basis for its topology consisting of clopen subsets. Then two points $x,y \in E$ are similar if and only if they are globally similar. Thus $x \in E$ is topologically distinguished if and only if it is globally topologically distinguished.
\end{lem}
\begin{proof}
The second statement follows from the first one, so we only have to prove the first statement, that $x,y \in E$ are similar if and only if they are globally similar. One implication is obvious; we will prove the opposite implication. So let us assume that $x,y \in E$ are similar and choose open neighbourhoods $U$ and $V$ of $x$ and $y$ respectively and a homeomorphism $\varphi \colon U \to V$ taking $x$ to $y$. Assume that $x\neq y$ (otherwise the result is obvious). Since $E$ is zero-dimensional, we may assume, by shrinking them if necessary, that $U$ and $V$ are clopen. Since $E$ is Hausdorff, we may assume, by shrinking them if necessary, that $U$ and $V$ are disjoint. We may therefore extend $\varphi$ to a homeomorphism $\bar{\varphi} \in \Homeo(E)$ by:
\begin{itemizeb}
\item $\bar{\varphi}(e) = \varphi(e)$ for $e \in U$;
\item $\bar{\varphi}(e) = \varphi^{-1}(e)$ for $e \in V$;
\item $\bar{\varphi}(e) = e$ for $e \in E \smallsetminus (U \sqcup V)$.
\end{itemizeb}
This bijection is continuous since $\{ U , V , E \smallsetminus (U \sqcup V) \}$ is an open cover of $E$ and $\bar{\varphi}$ is continuous when restricted to each of these subsets. Its inverse is continuous for the same reason, so it is a homeomorphism of $E$ taking $x$ to $y$. Thus $x$ and $y$ are globally similar.
\end{proof}

\begin{rem}
\label{rem:end-spaces-top-distinguished}
Ends spaces of surfaces are always Hausdorff and zero-dimensional, so Lemma \ref{lem:global-to-local} implies that \emph{topologically distinguished} and \emph{globally topologically distinguished} are the same for end spaces.
\end{rem}

\begin{lem}
\label{lem:top-distinguished-compactification}
If a space $E$ has a topologically distinguished point, then $E\omega+1$ has a globally topologically distinguished point. In fact, the point at infinity is globally topologically distinguished.
\end{lem}

\begin{proof}
Let $\infty$ denote the point at infinity of the one-point compactification $E\omega + 1$ of $E\omega = \bigsqcup_\omega E$. Let $\varphi \in \Homeo(E\omega + 1)$. We just need to show that $\varphi(\infty) = \infty$, since it will then follow that $\infty$ is a globally topologically distinguished point of $E\omega + 1$. Suppose for a contradiction that $\varphi(\infty) \neq \infty$. Write $E_i = E$ for each $i \in \bN$, and identify $E\omega = \bigsqcup_{i \in \bN} E_i$. By assumption, $\varphi(\infty) \in E_j$ for some $j \in \bN$. Let $x \in E$ be a topologically distinguished point. Every open neighbourhood $U$ of $\infty \in E\omega + 1$ contains infinitely many points that are similar to $x$, since, by definition of the one-point compactification, $U$ must contain $E_i$ for infinitely many $i$. Since $\varphi$ is a homeomorphism, it must also be true that every open neighbourhood of $\varphi(\infty) \in E\omega + 1$ contains infinitely many points that are similar to $x$. But $E_j$ is an open neighbourhood of $\varphi(\infty) \in E\omega + 1$ and it contains only one point that is similar to $x$, a contradiction.
\end{proof}

\begin{cor}
\label{cor:top-distinguished-compactification}
Suppose that $E$ is Hausdorff and zero-dimensional. If $E$ has a topologically distinguished point, then the point at infinity of $E\omega+1$ is topologically distinguished.
\end{cor}

\begin{proof}
By Lemma \ref{lem:top-distinguished-compactification}, the point at infinity of $E\omega+1$ is globally topologically distinguished. Hausdorffness and zero-dimensionality of $E$ imply Hausdorffness and zero-dimensionality of $E\omega+1$, so Lemma \ref{lem:global-to-local} then implies that the point at infinity of $E\omega+1$ is topologically distinguished.
\end{proof}

\begin{rem}
There is another, a priori different, equivalence relation on topological spaces, defined by \cite{MaRa22}. They define, for points $x,y \in E$:
\[
x \leq y \;\Longleftrightarrow\; \forall\text{ open neighbourhoods } U \ni y, \; \exists z \in U : z \sim x,
\]
where $z \sim x$ means that $z$ and $x$ are similar in the sense of Definition \ref{defn:topologically-distinguished}. This is a pre-order on $E$, so it induces an equivalence relation
\[
x \approx y \;\Longleftrightarrow\; x \leq y \text{ and } y \leq x
\]
on $E$ and a poset structure on the quotient $E/{\approx}$. Clearly $x \sim y$ implies $x \approx y$. Also, if we now assume that $E$ is the end space of a surface $\Sigma$, it is not hard to see (using Lemma \ref{lem:global-to-local} and Proposition \ref{thm:ses-bmcg}) that $x \sim y$ if and only if there is a homeomorphism of $\Sigma$ taking $x$ to $y$. Theorem~1.2 of \cite{MaRa22} says that if $x \approx y$ then there is a homeomorphism of $\Sigma$ taking $x$ to $y$. It follows that $\sim$ and $\approx$ are the same equivalence relation on $E$ if it is the end space of a surface. In \cite{MannRafi}, the authors often consider the condition that ``$\Sigma$ has a unique maximal end'', i.e.~there is a unique maximal equivalence class $[x] \in E/{\approx}$ and the equivalence class $[x]$ has size $1$. The condition that we require in this paper is however much weaker, namely that ``$\Sigma$ has a topologically distinguished end'', i.e.~there is an equivalence class $[x] \in E/{\approx}$ of size $1$ (but it need not be maximal in the poset structure of $E/{\approx}$).
\end{rem}

\section{Tools for proving uncountability}
\label{s:lifting-uncountability}

We start with a key lemma, which we use several times to conclude uncountability of the homology of a given group $G$ in all positive degrees.

\begin{notation}
Let us fix some notation that will be used throughout the rest of the paper.
\begin{itemizeb}
\item For an abelian group $A$, denote by $\Lambda^*(A)$ the exterior algebra on $A$. 
\item We denote by $\fc$ the cardinality of the continuum.
\end{itemizeb}
\end{notation}

\begin{lem}
\label{lem:lift-uncount}
Let $G$ be a group, denote by $\alpha \colon G\twoheadrightarrow G^{ab} = H_1(G)$ the quotient onto its abelianisation and let $\iota \colon \bigoplus_{\fc} \bZ \to G$ be a homomorphism. Suppose that there is an embedding $f \colon \bigoplus_{\fc} \bQ \hookrightarrow H_1(G)$ such that the diagram
\begin{equation}
\label{eq:uncount-lemm}
\begin{tikzcd}
\displaystyle\bigoplus_{\fc} \bZ \ar[r,"\iota"] \ar[d,hook] & G \ar[d,two heads,"\alpha"] \\
\displaystyle\bigoplus_{\fc} \bQ \ar[r,"f",hook] & H_1(G),
\end{tikzcd}
\end{equation}
commutes, where $\bigoplus_{\fc} \bZ \hookrightarrow \bigoplus_{\fc} \bQ$ is the canonical inclusion. Then there is an injective homomorphism of graded abelian groups
\[
\Lambda^*\Bigl(\bigoplus_{\fc} \bZ\Bigr) \longhookrightarrow H_*(G).
\]
In particular, $H_i(G)$ is uncountable for all $i\geq 1$.
\end{lem}

\begin{proof}
By Lemma \ref{lem:divisible-injective}, the embedding $f$ admits a retraction. Hence the canonical inclusion
\begin{equation}
\label{eq:inclusion-Z-to-Q}
\bigoplus_{\fc} \bZ \longhookrightarrow \bigoplus_{\fc} \bQ
\end{equation}
factors through $G$. It follows that the induced homomorphism of graded abelian groups
\begin{equation}
\label{eq:inclusion-Z-to-Q-on-homology}
H_*\Bigl(\bigoplus_{\fc} \bZ\Bigr) \longrightarrow H_*\Bigl(\bigoplus_{\fc} \bQ\Bigr)
\end{equation}
factors through $H_*(G)$. The integral homology of a torsion-free abelian group $A$ is naturally isomorphic to the exterior algebra $\Lambda^*(A)$ (see \cite[Theorem~V.6.4(ii)]{Brown1982}), so we have homomorphisms of graded abelian groups
\begin{equation}
\label{eq:inclusion-Z-to-Q-exterior}
\Lambda^*\Bigl(\bigoplus_{\fc} \bZ\Bigr) \longrightarrow H_*(G) \longrightarrow \Lambda^*\Bigl(\bigoplus_{\fc} \bQ\Bigr)
\end{equation}
whose composition is injective by Lemma \ref{lem:exterior-Z-to-Q}. In particular the first map must be injective.
\end{proof}

In order to apply Lemma \ref{lem:lift-uncount}, we will need to be able to construct embeddings of direct sums of copies of $\bQ$ into the first homology of big mapping class groups. The key topological input for this is a theorem of Domat, which we recall below and whose proof uses the machinery of Bestvina, Bromberg and Fujiwara \cite{BestvinaBromgergFujiwara2015}. We first make some definitions that are implicit in the statement of \cite[Theorem~6.1]{Domat2022}.

\begin{defn}
\label{defn:escaping-sequence}
Let $S$ be a connected surface with at least two ends. Let us call a sequence $\{\gamma_i\}_{i \in \bN}$ of isotopy classes of simple closed curves on $S$ an \emph{escaping sequence} if:
\begin{itemizeb}
\item each $\gamma_i$ is end-separating, i.e., cutting along it disconnects $S$ into two non-compact surfaces;
\item $\gamma_i$ and $\gamma_j$ have pairwise-disjoint representatives for $i\neq j$;
\item the sequence $\gamma_1,\gamma_2,\ldots$ eventually leaves every compact subset of $S$, i.e., if $K \subset S$ is a compact subset then only finitely many $\gamma_i$ may be isotoped to lie in $K$.
\end{itemizeb}
An escaping sequence $\{\gamma_i\}_{i \in \bN}$ is \emph{well-spaced} if there exists another escaping sequence $\{\gamma'_i\}_{i \in \bN}$ such that:
\begin{itemizeb}
\item $\gamma'_i$ is not isotopic to $\gamma_i$;
\item $\gamma'_i$ and $\gamma_j$ have pairwise-disjoint representatives for $i\neq j$;
\item there is a (necessarily non-trivial) element $g_i \in \PMap_c(S)$ taking $\gamma_i$ to $\gamma'_i$.
\end{itemizeb}
\end{defn}

\begin{rem}
\label{rem:escaping-sequence-existence}
It follows from the classification of surfaces that an escaping sequence exists on $S$ if and only if $S$ has infinite type. In addition, any escaping sequence becomes well-spaced after passing to an appropriate subsequence.
\end{rem}

\begin{ex}
\label{ex:punctured-Loch-Ness}
In the key example of $S=L'$ the once-punctured Loch Ness monster surface, we may for example take $\{\gamma_i\}_{i \in \bN}$ to be the sequence of curves pictured in Figure \ref{fig:punctured-Loch-Ness}. Each $\gamma_i$ is clearly end-separating, they are pairwise disjoint and no compact subset of $L'$ contains more than finitely many of them, so this sequence is \emph{escaping}. Moreover, taking $\gamma'_i = T_{\alpha_i}(\gamma_i)$ using the curves $\alpha_i$ also pictured in Figure \ref{fig:punctured-Loch-Ness}, we obtain another escaping sequence $\{\gamma'_i\}_{i \in \bN}$ witnessing that $\{\gamma_i\}_{i \in \bN}$ is \emph{well-spaced}.
\end{ex}

\begin{ex}
\label{ex:flute-surface}
As another example, we may consider the \emph{flute surface} depicted in Figure \ref{fig:flute-surface}, together with the curves $\gamma_i$ illustrated. These form an escaping sequence $\{\gamma_i\}_{i \in \bN}$, but this is \emph{not} a \emph{well-spaced} escaping sequence: for example, one may attempt to construct another escaping sequence witnessing that it is well-spaced by setting $\gamma'_i = T_{\alpha_i}(\gamma_i)$ using the curves $\alpha_i$ illustrated, but then $\gamma'_i$ intersects $\gamma'_{i+1}$, so $\{\gamma'_i\}_{i \in \bN}$ is not an escaping sequence as in Definition \ref{defn:escaping-sequence}. However, the subsequence $\{\gamma_{2i}\}_{i\in\bN}$ \emph{is} well-spaced, as witnessed by the subsequence $\{\gamma'_{2i}\}_{i\in\bN}$.
\end{ex}

\begin{figure}[tb]
    \centering
    \includegraphics[scale=0.6]{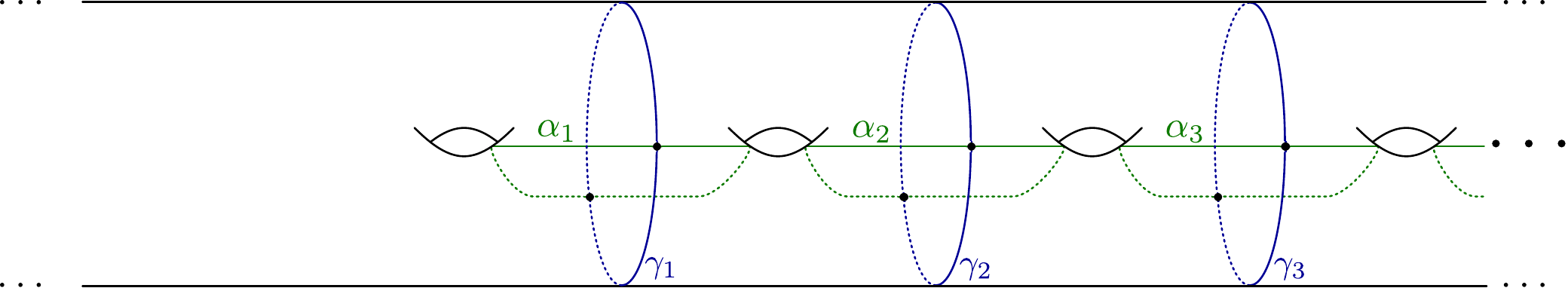}
    \caption{The once-punctured Loch Ness monster surface equipped with a sequence $\{\gamma_i\}_{i\in\bN}$ of simple closed curves that is a \emph{well-spaced}, \emph{escaping} sequence in the sense of Definition \ref{defn:escaping-sequence}. The fact that it is well-spaced is witnessed by the associated sequence of simple closed curves $\{\gamma'_i\}_{i\in\bN}$ given by $\gamma'_i = T_{\alpha_i}(\gamma_i)$.}
    \label{fig:punctured-Loch-Ness}
    \vspace{1em}
    \includegraphics[scale=0.6]{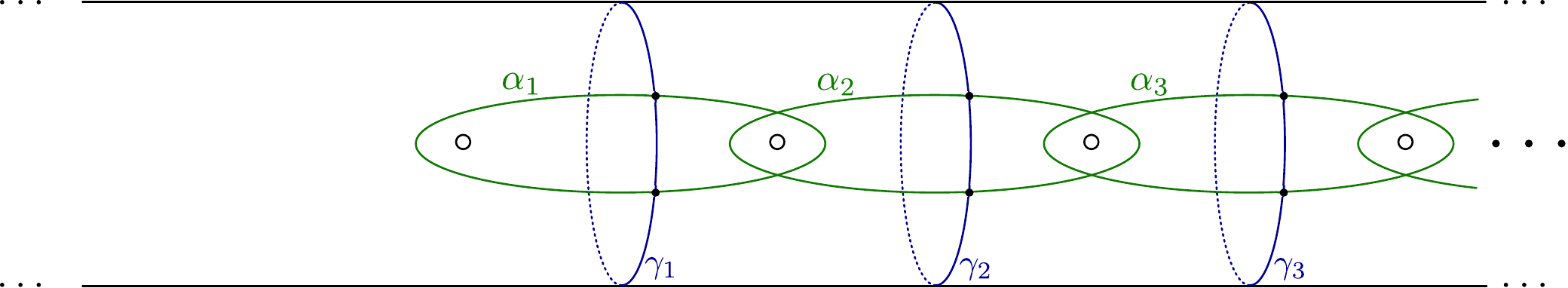}
    \caption{The flute surface equipped with a sequence $\{\gamma_i\}_{i\in\bN}$ of simple closed curves that is an \emph{escaping} sequence in the sense of Definition \ref{defn:escaping-sequence}. After passing to the subsequence $\{\gamma_{2i}\}_{i\in\bN}$, it becomes well-spaced, as explained in Example \ref{ex:flute-surface}.}
    \label{fig:flute-surface}
\end{figure}

\begin{thm}[{\cite[Theorem~6.1]{Domat2022}}]
\label{thm:Domat}
Let $S$ be an infinite-type surface with at least two ends and let $\{\gamma_i\}_{i \in \bN}$ be a well-spaced escaping sequence of simple closed curves on $S$. Let $a_1,a_2,\ldots$ be an unbounded sequence of positive integers. Then
\[
\prod_{i=1}^\infty (T_{\gamma_i})^{a_i} \in \overline{\PMap_c(S)}
\]
projects to a non-zero element in $\bigl( \overline{\PMap_c(S)} \bigr)^{ab}$.
\end{thm}

In fact, what is used in practice in \cite{Domat2022} is the following stronger fact, in the case when $S$ has genus at least three. It is implicit in \cite[\S 8.1.1]{Domat2022}; here we make the statement and the details of the proof explicit.

\begin{cor}
\label{cor:embedding-of-Q}
Let $S$ be an infinite-type surface of genus at least three with at least two ends and let $\{\gamma_i\}_{i \in \bN}$ be a well-spaced escaping sequence of simple closed curves on $S$. Let $a_1,a_2,a_3,\ldots$ be a strictly increasing sequence of positive integers. Then there is an injective homomorphism
$\varphi \colon \bQ \hookrightarrow \bigl( \overline{\PMap_c(S)} \bigr)^{ab}$
sending $1/n \in \bQ$ to the element
\[
\prod_{i=r_n}^\infty (T_{\gamma_i})^{a_i!/n} \in \bigl( \overline{\PMap_c(S)} \bigr)^{ab},
\]
where $r_n\geq 1$ is any integer sufficiently large so that $a_i\geq n$ for all $i\geq r_n$.
\end{cor}

\begin{proof}
Using the presentation $\bQ \cong \langle x_1, x_2, x_3, \ldots | (x_n)^n = x_{n-1} \rangle$, where $x_n$ corresponds to $1/n! \in \bQ$, we see that in order to define a homomorphism $\varphi \colon \bQ \to G$, for any group $G$, it suffices to choose an element $\varphi(1)$ of $G$, a square root $\varphi(1/2!)$ of $\varphi(1)$, a cube root $\varphi(1/3!)$ of $\varphi(1/2!)$, etc. We begin by choosing
\[
\varphi(1) = \prod_{i=1}^\infty (T_{\gamma_i})^{a_i!} \in \bigl( \overline{\PMap_c(S)} \bigr)^{ab}.
\]
This is non-trivial by Theorem \ref{thm:Domat}, since the sequence $(a_i!)$ is unbounded. In fact, Theorem \ref{thm:Domat} implies that $\varphi(1)$ has infinite order, since the sequence $(na_i!)$ is unbounded for all $n\geq 1$. We next need to choose a square root $\varphi(1/2!)$ of this element. First choose $r_2 \geq 1$ so that $a_i \geq 2$ for all $i\geq r_2$ (this is possible since $(a_i)$ is strictly increasing). Then set
\[
\varphi(1/2!) = \prod_{i=r_2}^\infty (T_{\gamma_i})^{a_i!/2!} \in \bigl( \overline{\PMap_c(S)} \bigr)^{ab}
\]
and notice that
\[
\frac{\varphi(1)}{2\varphi(1/2!)} = \prod_{i=1}^{r_2 - 1} (T_{\gamma_i})^{a_i!} \in \bigl( \overline{\PMap_c(S)} \bigr)^{ab}.
\]
This is a finite product of Dehn twists, so it is the image of the corresponding element of $\PMap_c(S)^{ab}$. Restricting further, choose a compact subsurface $S' \subset S$ containing the curves $\gamma_1,\ldots,\gamma_{r_2 - 1}$ in its interior and having genus at least three. The element above is then the image of the corresponding element of $\Map(S')^{ab}$. But the mapping class group of any compact, orientable surface of genus at least three is perfect~\cite{Birman1970,Powell78}, so $\Map(S')^{ab} = 0$ and hence $\varphi(1) = 2\varphi(1/2!)$. Continuing in the same way, we construct a cube root of $\varphi(1/2!)$, etc. Thus we have constructed a homomorphism $\varphi$ from $\bQ$.

Recall that any homomorphism defined on $\bQ$ is injective as long as its restriction to $\bZ \subset \bQ$ is injective. We observed above that $\varphi(1)$ has infinite order; hence $\varphi$ is injective. Finally, the formula for $\varphi(1/n)$ in the statement follows immediately from the construction, noting again that we may remove finitely many terms from the infinite product without changing the element of the abelianisation.
\end{proof}

The following corollary is again implicit in \cite[\S 8.1.1]{Domat2022}, but we prefer to make the statement and the details of the proof explicit. Let the surface $S$ and the sequences $\{\gamma_i\}_{i \in \bN}$ and $\{a_i\}_{i\in\bN}$ be as in Corollary \ref{cor:embedding-of-Q}. For any infinite subset $F \subseteq \bN$, denote by
\[
\varphi_F \colon \bQ \longhookrightarrow \bigl( \overline{\PMap_c(S)} \bigr)^{ab}
\]
the embedding obtained by applying Corollary \ref{cor:embedding-of-Q} to the sequences $\{\gamma_i\}_{i \in \bN}$ and $\{a_i\}_{i\in\bN}$.

\begin{cor}
\label{cor:embedding-of-QQ}
Let $\cF$ be a family of infinite subsets of $\bN$ such that any two of them have finite intersection. Then the homomorphism
\[
\Phi_\cF = \bigoplus_{F \in \cF} \varphi_F \colon \bigoplus_{F \in \cF} \bQ \longrightarrow \bigl( \overline{\PMap_c(S)} \bigr)^{ab}
\]
is also injective.
\end{cor}
\begin{proof}
Let $(r_F) \in \mathrm{ker}(\Phi_\cF)$. Since the domain of $\Phi_\cF$ is a direct sum, there are only finitely many $F \in \cF$ such that $r_F \neq 0$; let us enumerate these as $F_1,\ldots,F_s$. Also choose $n\geq 1$ so that $m_F := nr_F \in \bZ$. We therefore have
\[
0 = \Phi_\cF(n(r_F)) = \Phi_\cF((m_F)) = \prod_{i \in F_1} \left( (T_{\gamma_i})^{a_i!} \right)^{m_{F_1}} \cdots \prod_{i \in F_s} \left( (T_{\gamma_i})^{a_i!} \right)^{m_{F_s}}.
\]
By Theorem \ref{thm:Domat}, this product can only be zero if it is a \emph{finite} product. But each $F_1,\ldots,F_s$ is infinite. Moreover, two terms of the product can only cancel if they are indexed by an element of one of the pairwise intersections $F_p \cap F_q$ for $p \neq q \in \{1,\ldots,s\}$, all of which are finite by assumption. Thus only finitely many cancellations can occur, so the only possible way for this product to be zero is if $s=0$, which means that $(r_F) = 0$. Thus $\Phi_\cF$ is injective.
\end{proof}

\section{Proof of Theorem \ref{athm:PMap}}
\label{s:Loch-Ness}

We are now ready to prove Theorem \ref{athm:PMap}. The tools of the previous section imply almost immediately the following result.

\begin{prop}
\label{prop:closure-of-compactly-supported-mcg}
Let $S$ be an infinite type surface of genus at least three with at least two ends. Then there is an injective homomorphism of graded abelian groups
\[
\Lambda^*\Bigl(\bigoplus_{\fc} \bZ\Bigr) \longhookrightarrow H_*(\overline{\PMap_c(S)}).
\]
\end{prop}
\begin{proof}
Choose a well-spaced escaping sequence $\{\gamma_i\}_{i \in \bN}$ of simple closed curves on $S$ (such a sequence always exists by Remark~\ref{rem:escaping-sequence-existence}) and set $a_i = i$. Choose a family $\cF$ of infinite subsets of $\bN$ such that any two of them have finite intersection, and such that the family $\cF$ has the cardinality of the continuum. (For example, we may identify $\bN$ with $\bQ$ and choose for each $a \in \bR$ a sequence of distinct rationals converging to $a$.) There is then a commutative diagram
\begin{equation}
\label{eq:square-for-closure-of-PMapc}
\begin{tikzcd}
\displaystyle\bigoplus_{\cF} \bZ \ar[r] \ar[d,hook] & \overline{\PMap_c(S)} \ar[d,two heads] \\
\displaystyle\bigoplus_{\cF} \bQ \ar[r,"\Phi_\cF",hook] & \bigl(\overline{\PMap_c(S)}\bigr)^{ab},
\end{tikzcd}
\end{equation}
where the bottom horizontal map $\Phi_\cF$ is injective by Corollary \ref{cor:embedding-of-QQ} and its lift to $\overline{\PMap_c(S)}$ after restricting to $\bZ \subset \bQ$ in each summand is given by sending the generator $1 \in \bZ$ of the summand corresponding to $F \in \cF$ to the element
\[
\prod_{i \in F} (T_{\gamma_i})^{i!} \in \overline{\PMap_c(S)}.
\]
The result then follows by an application of Lemma \ref{lem:lift-uncount}.
\end{proof}

\begin{rem}
\label{rmk:closure-of-compactly-supported-mcg}
Proposition \ref{prop:closure-of-compactly-supported-mcg} also holds without the assumption that $S$ has genus at least $3$. This follows from an analogue of Corollary \ref{cor:embedding-of-Q} that involves a sequence of pseudo-Anosov elements supported on pairwise-disjoint compact subsurfaces of $S$ instead of Dehn twists; see \cite[\S 8.1.2]{Domat2022} for more details of this construction. One then obtains a diagram of the form \eqref{eq:square-for-closure-of-PMapc}, where the horizontal maps are defined using infinite products of powers of these pseudo-Anosov elements instead of Dehn twists, and the result then follows from Lemma \ref{lem:lift-uncount}.
\end{rem}

We next deduce the analogue of Theorem \ref{athm:uncountable} for the Loch Ness monster surface $L$ and the surface $L'$ obtained by removing one puncture from $L$.

\begin{prop}
\label{prop:uncountability-Loch-Ness}
The graded abelian groups $H_*(\Map(L'))$ and $H_*(\Map(L))$ each contain an embedded copy of the exterior algebra $\Lambda^*\bigl(\bigoplus_{\fc} \bZ\bigr)$.
\end{prop}
\begin{proof}
Since $L'$ has at most one non-planar end, \cite[Theorem 4]{PriyamVlamis18} implies that $\overline{\PMap_c(L')} = \PMap(L')$. We also have $\PMap(L') = \Map(L')$ since $L'$ has only two punctures, which cannot be interchanged by a homeomorphism of $L'$ since exactly one of them is non-planar. Thus the result for $L'$ is a special case of Proposition \ref{prop:closure-of-compactly-supported-mcg}. In this case, the sequence of simple closed curves $\gamma_i$ may be taken to be those illustrated in Figure \ref{fig:punctured-Loch-Ness} (see Example \ref{ex:punctured-Loch-Ness}).

In order to deduce the result for $L$, we use the Birman exact sequence, which takes the form
\begin{equation}
\label{eq:Birman-exact-ses}
1 \to \pi_1(L) \longrightarrow \Map(L') \longrightarrow \Map(L) \to 1.
\end{equation}
Since abelianisation is a right-exact functor, it follows that the kernel of $H_1(\Map(L')) \to H_1(\Map(L))$ is a quotient of $H_1(L)$; in particular it is \emph{countable}. Consider the diagram
\begin{equation}
\label{eq:induced-square-extended}
\begin{tikzcd}
\displaystyle\bigoplus_{\cF} \bZ \ar[r] \ar[d,hook] & \Map(L') \ar[d,two heads] \ar[r] & \Map(L) \ar[d,two heads] \\
\displaystyle\bigoplus_{\cF} \bQ \ar[r,"\Phi_\cF",hook] & (\Map(L'))^{ab} \ar[r,"(*)"] & (\Map(L))^{ab}
\end{tikzcd}
\end{equation}
where the left-hand square is \eqref{eq:square-for-closure-of-PMapc} in the case $S = L'$ and the right-hand square is induced by \eqref{eq:Birman-exact-ses}. We know that $(*)$ has countable kernel by the discussion above, so Lemma \ref{lem:countable-kernel-injective} implies that, after removing countably many terms from the direct sum on the left-hand side, the composition across the bottom of \eqref{eq:induced-square-extended} is also injective. We therefore obtain a diagram
\begin{equation}
\label{eq:induced-square-modified}
\begin{tikzcd}
\displaystyle\bigoplus_{\fc} \bZ \ar[r] \ar[d,hook] & \Map(L) \ar[d,two heads] \\
\displaystyle\bigoplus_{\fc} \bQ \ar[r,"{(*)'}",hook] & (\Map(L))^{ab},
\end{tikzcd}
\end{equation}
where $(*)'$ is injective and the direct sums on the left-hand side are still indexed by a set with the cardinality of the continuum. The result for $L$ thus follows from Lemma \ref{lem:lift-uncount}.
\end{proof}

\begin{rem}
\label{rmk:closure-of-compactly-supported-mcg-2}
We noted in Remark \ref{rmk:closure-of-compactly-supported-mcg} that Proposition \ref{prop:closure-of-compactly-supported-mcg} holds without the assumption on the genus of $S$, i.e.~it holds for any infinite type surface $S$ with at least two ends. On the other hand, if $S$ is an infinite type surface with at most one end, it must be the Loch Ness monster surface $S=L$, and the result then follows from Proposition \ref{prop:uncountability-Loch-Ness} (see also \cite[Appendix]{Domat2022}). Thus, in fact, Proposition \ref{prop:closure-of-compactly-supported-mcg} holds for any infinite type surface $S$. This is the first part of Theorem \ref{athm:PMap}:
\end{rem}

\begin{cor}
\label{cor:closure-of-compactly-supported-mcg}
Let $S$ be an infinite-type surface. Then the graded abelian group $H_*(\overline{\PMap_c(S)})$ contains an embedded copy of the exterior algebra $\Lambda^*\bigl(\bigoplus_{\fc} \bZ\bigr)$, induced by an embedding $\bigoplus_{\fc} \bZ \hookrightarrow \overline{\PMap_c(S)}$.
\end{cor}

\begin{rem}
There are two points where this proof is not entirely constructive. The first is the choice of the family $\cF = \{ \Lambda_a \mid a \in \bR \}$ of infinite subsets of $\bN$. However, this may easily be made explicit by choosing an explicit bijection between $\bN$ and $\bQ$ and then letting $\Lambda_a \subseteq \bQ$, for $a \in \bR$, be the sequence of rational numbers converging to $a \in \bR$ given by truncating the binary expansion of $a$. The second point where it is non-constructive is in passing from diagram \eqref{eq:induced-square-extended} to diagram \eqref{eq:induced-square-modified} by throwing away countably many real numbers indexing the direct sum on the left-hand side. However, looking carefully at the proof of Lemma \ref{lem:countable-kernel-injective}, one may make this step constructive too.
\end{rem}

\begin{rem}
\label{rem:hmg-PMap-uncountable}
When $S$ has at most one non-planar end, the pure mapping class group $\PMap(S)$ coincides with $\overline{\PMap_c(S)}$, by \cite[Theorem~4]{PriyamVlamis18}. Thus Corollary \ref{cor:closure-of-compactly-supported-mcg} says that $H_\ast(\PMap(S))$ is uncountable in every positive degree when $S$ has at most one non-planar end. This statement also holds when $S$ has infinitely many non-planar ends. Indeed, by \cite[Corollary 6]{APV20}, we have in this case that
\[
\PMap(S) \;\cong\; \overline{\PMap_c(S)} \rtimes \BZ^{\BN}.
\]
In particular, $\BZ^{\BN}$ is a retract of $\PMap(S)$, so the natural induced map $H_i(\BZ^{\BN}) \to H_i(\PMap(S))$ is split-injective in every degree. The fact that that $H_\ast(\PMap(S))$ is uncountable in every positive degree in this case is therefore an immediate corollary of the following lemma.
\end{rem}

\begin{lem}
\label{lem:Baer-Specker}
The homology group $H_i(\BZ^{\BN})$ contains a direct summand isomorphic to $\BZ^{\BN}$ in every degree $i > 0$. Hence it contains a subgroup isomorphic to $\bigoplus_{\fc} \bZ$ in every degree $i > 0$. 
\end{lem}
\begin{proof}
The first statement follows from the K{\"u}nneth theorem applied to the decomposition $\BZ^\BN \cong \BZ^\BN \times \BZ^i$. The second statement then follows from the fact that $\BZ^{\BN}$ contains free abelian groups of rank $\fc$. To see this, choose a family $\cF$, of cardinality $\lvert \cF \rvert = \fc$, of infinite subsets of $\bN$ such that any pair have finite intersection. (For example, as in the proof of Proposition \ref{prop:closure-of-compactly-supported-mcg}, we may identify $\bN$ with $\bQ$ and choose for each $a \in \bR$ a sequence of distinct rationals converging to $a$.) It is then easy to check that the collection
\[
\{ \chi_F \in \bZ^\bN \mid F \in \cF \},
\]
where $\chi_F \colon \bN \to \{0,1\} \subset \bZ$ denotes the indicator function of $F \subseteq \bN$, is $\bZ$-linearly independent and hence generates a subgroup of $\bZ^\bN$ isomorphic to $\bigoplus_\fc \bZ$.
\end{proof}

\begin{rem}
\label{rem:gap-pmap-hmg}
When $S$ has $n$ non-planar ends with $1< n<\infty$, by \cite[Corollary 6]{APV20} we have:
\begin{equation}
\label{eq:semi-direct-product-finite}
\PMap(S) \;\cong\; \overline{\PMap_c(S)} \rtimes \BZ^{n-1},
\end{equation}
where $\BZ^{n-1}$ is freely generated by $n-1$ handle shifts $h_1,\ldots, h_{n-1}$. As indicated in the proof of \cite[Theorem 5]{APV20}, one may choose the handle shifts $h_j$  to have pairwise disjoint support. Let $Y_j$ be the support of $h_j$. Recall that each $Y_j$ is a subsurface homeomorphic to the result of gluing handles onto $\BR \times [0,1]$ periodically with respect to the transformation $(x,y) \mapsto (x+1,y)$. For convenience, we shall require that the $i$-th handle is attached to $[i,i+1]\times [0,1]$ and that $h_j$ maps the $i$-th handle to the $(i+1)$-st handle. See Figure \ref{fig:handle-shifts-and-Dehn-twists} for an illustration. The semi-direct product decomposition \eqref{eq:semi-direct-product-finite} implies that
\begin{equation}
\label{eq:semi-direct-product-abelianisation}
H_1(\PMap(S)) \;\cong\; H_1(\overline{\PMap_c(S)})_{\bZ^{n-1}} \oplus \bZ^{n-1},
\end{equation}
where $(-)_{\bZ^{n-1}}$ denotes the coinvariants under the action of the handle shifts. By Theorem \ref{thm:Domat}, choosing the sequence of curves $\gamma_i$ as illustrated in Figure \ref{fig:handle-shifts-and-Dehn-twists} and any unbounded sequence of positive integers $a_i$, the infinite product of Dehn twists $f = \prod_{i=1}^{\infty} (T_{\gamma_i})^{a_i} \in \overline{\PMap_c(S)}$ represents a non-trivial element in the abelianisation $H_1(\overline{\PMap_c(S)})$. But it vanishes in $H_1(\PMap(S))$ -- in other words, in the coinvariants under the action of the handle shifts -- since $[f] = [g] - [h_1gh_1^{-1}]$, where $g = \prod_{i=1}^{\infty} (T_{\gamma_i})^{b_i}$, $b_i = \Sigma_{j=1}^i a_j$.
\end{rem}

\begin{figure}
    \centering
    \includegraphics{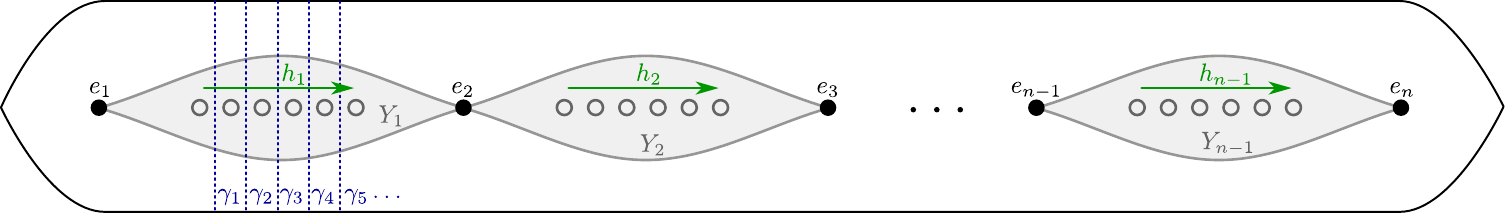}
    \caption{A surface with $n$ non-planar ends $e_1,\ldots,e_n$ for $2\leq n<\infty$. The top and bottom edges are identified to obtain a sphere, then the points $e_1,\ldots,e_n$ (together with a set of planar ends, which is not pictured) are removed, then we take a connected sum with a torus along each of the (infinitely many) small grey discs. The planar ends (not pictured) may have some or all of the non-planar ends $e_1,\ldots,e_n$ as limit points, but in any case lie \emph{outside} of the subsurfaces $Y_1,\ldots,Y_{n-1}$, which support the handle shifts $h_1,\ldots,h_{n-1}$. The curves $\gamma_1,\gamma_2,\gamma_3,\ldots$ are chosen as illustrated such that the handle shift $h_1$ sends $\gamma_i$ to $\gamma_{i+1}$ (up to isotopy).}
    \label{fig:handle-shifts-and-Dehn-twists}
\end{figure}

Recall that the Torelli group $\mathcal{T}(S)$ is the kernel of the natural homomorphism $\Map(S)\to \Aut(H_1(S))$.

\begin{thm}
\label{thm:torelli}
Let $S$ be an infinite-type surface. The integral homology group $H_i(\mathcal{T}(S))$ is uncountable for every $i \geq 1$. In fact it contains an embedded copy of $\bigoplus_{\fc} \bZ$ in every positive degree.
\end{thm}

\begin{proof}
By Corollary \ref{cor:closure-of-compactly-supported-mcg}, there is an embedding
\begin{equation}
\label{eq:embedding-of-Zc}
\textstyle\bigoplus_{\fc} \bZ \longhookrightarrow \overline{\PMap_c(S)}
\end{equation}
that induces on homology an embedding of $\Lambda^*\bigl(\bigoplus_{\fc} \bZ\bigr)$ into $H_*(\overline{\PMap_c(S)})$. It will therefore suffice to show that \eqref{eq:embedding-of-Zc} factors through the Torelli group $\mathcal{T}(S)$.

We first note that the Torelli group is contained in $\overline{\PMap_c(S)} \subset \Map(S)$: it clearly lies in $\PMap(S)$ since any non-trivial action on the space of ends of $S$ implies a non-trivial action on $H_1(S)$; then the fact that it lies in $\overline{\PMap_c(S)}$ follows from \cite[Corollary~6]{APV20}, which decomposes $\PMap(S)$ as a semi-direct product of $\overline{\PMap_c(S)}$ and a direct product of copies of $\bZ$ generated by \emph{handle shifts}, together with the fact that handle shifts act non-trivially on $H_1(S)$.

Finally, we just have to note that the elements of $\overline{\PMap_c(S)}$ used to define the homomorphism \eqref{eq:embedding-of-Zc} actually lie in $\mathcal{T}(S)$. When the genus of $S$ is at least $3$, these elements are infinite products of Dehn twists around (pairwise disjoint) \emph{separating} curves; hence they act trivially on $H_1(S)$. When the genus is at most $2$, we instead use infinite products of (pairwise disjointly-supported) pseudo-Anosov elements, as explained in Remark \ref{rmk:closure-of-compactly-supported-mcg}. These elements are of the form $T^2_\alpha T^2_\beta T^{-2}_\alpha T^{-2}_\beta$ for a pair of separating curves $\alpha,\beta$ that fill a finite-type subsurface of $S$, as explained in \cite[p.~715]{Domat2022}, and they also act trivially on $H_1(S)$.
\end{proof}

\begin{rem}
In degree one, $H_1(\mathcal{T}(S))$ contains an embedded copy of $\bigoplus_\fc \bQ$, by \cite[Theorem~9.1]{Domat2022}.
\end{rem}

\section{Descending along double branched covers}
\label{s:double-branched-covers}

In this section we generalise techniques of Malestein and Tao~\cite{MalesteinTao21} --- who proved uncountability of homology in degree $1$ for the mapping class group of $\bR^2 \smallsetminus \bN$ --- to higher degrees and to the more general class of surfaces from Theorem~\ref{athm:uncountable}, completing the proof of that theorem. To do this, we will need the notion of a \emph{ray surface} associated to a surface $\Sigma$.

\begin{defn}
\label{defn:ray-surface}
Let $\Sigma$ be any connected surface without boundary and write $\Sigma_1$ (respectively $\Sigma_2$) for the surface obtained by removing one (respectively two disjoint) open discs from $\Sigma$. The \emph{ray surface} $\cR(\Sigma)$ is the surface obtained by gluing together infinitely many copies of $\Sigma_2$ and ``capping off'' in one direction with a single copy of $\Sigma_1$. See the top half of Figure \ref{fig:double-covering} for an example where $\Sigma = T^2$ is the torus; thus $\cR(T^2)$ is the Loch Ness monster surface.
\end{defn}

\begin{figure}[tb]
    \centering
    \includegraphics[scale=0.8]{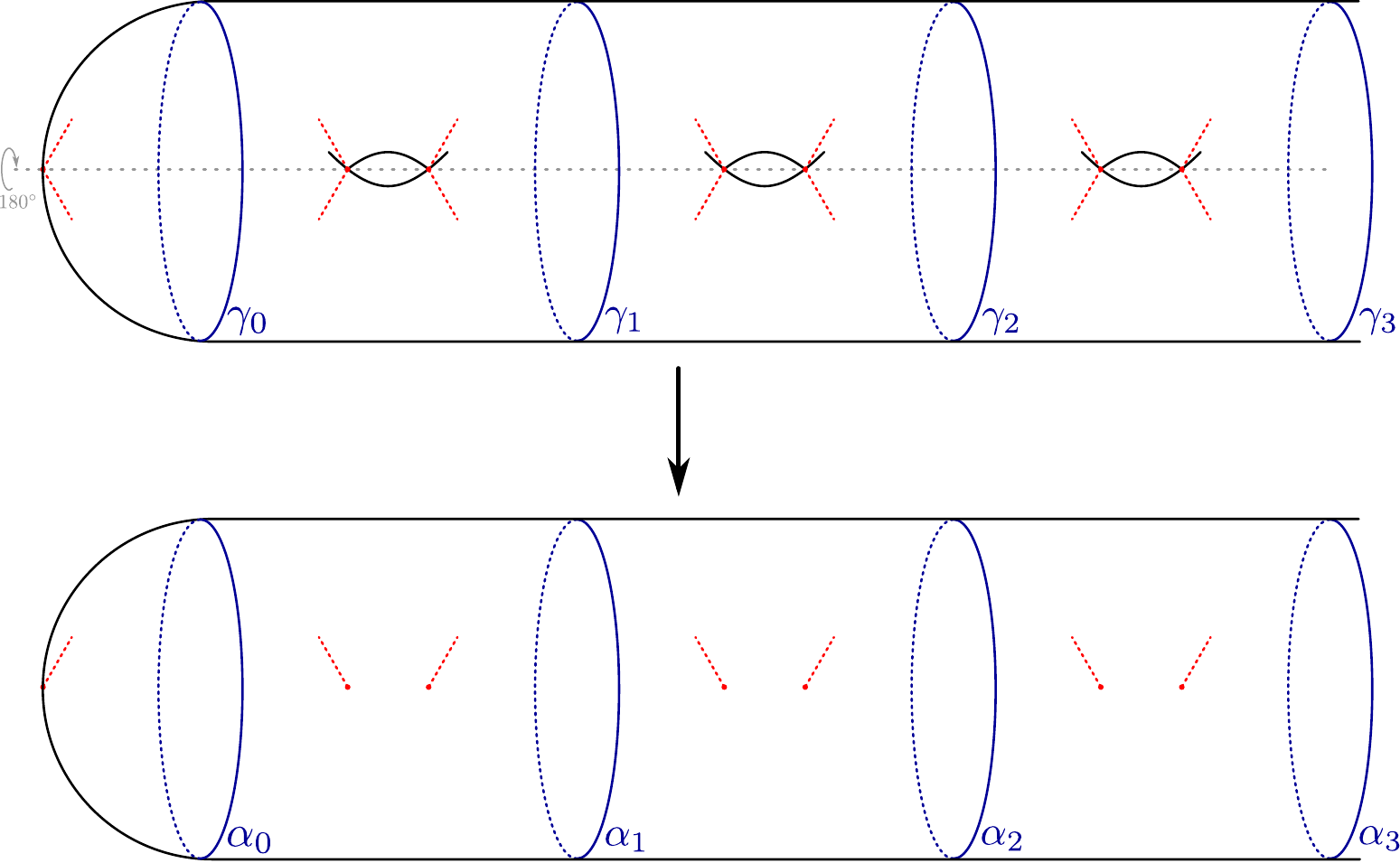}
    \caption{The branched double covering \eqref{eq:branched-covering}. After removing the subset marked in red (which includes the branch points), this restricts to the (genuine) double covering \eqref{eq:double-covering}. }
    \label{fig:double-covering}
\end{figure}

\begin{rem}
This is the same as the surface denoted by $\mathfrak{L}(\Sigma)$ in \cite{PalmerWu1} with its boundary capped off by a disc.
\end{rem}

Before proving Theorem~\ref{athm:uncountable} in general, we will prove it under certain stronger hypotheses on the surface $S$. Namely, we assume that the surface $S$ has genus $0$, empty boundary and that its space of ends is of the form $\Upsilon^+(E)$,\footnote{Recall that the notation $\Upsilon^+(-)$ was defined in Definition \ref{defn:Upsilon}.} where $E$ has a topologically distinguished point. This means that $S$ may be written as $\cR(\bS^2 \smallsetminus E)$, using the construction $\cR(-)$ of ray surfaces from Definition \ref{defn:ray-surface} above.

Denote by $L$ the Loch Ness monster surface and consider its branched double covering $L \to \bR^2$ depicted in Figure \ref{fig:double-covering}. This may also be written as
\begin{equation}
\label{eq:branched-covering}
L \cong \bS^2 \sharp \cR(T^2) \longrightarrow \bS^2 \sharp \cR(\bS^2) \cong \bR^2 .
\end{equation}
This decomposition corresponds to cutting along the curves depicted in the figure. Notice that there are exactly two branch points (of order $2$) in each copy of $\bS^2$ in $\cR(\bS^2)$ and one additional branch point in the copy of $\bS^2$ in the extra connected summand. Let us now choose once and for all a topologically distinguished point $x \in E$ (this exists by hypothesis) and embed pairwise disjoint copies of $E$ into $\bS^2 \sharp \cR(\bS^2)$ so that:
\begin{itemizeb}
\item each copy of $E$ lies entirely in one of the copies of $\bS^2$,
\item the point $x \in E$ is sent to a branch point of \eqref{eq:branched-covering},
\item each branch point of \eqref{eq:branched-covering} is in the image of one of the embeddings of $E$.
\end{itemizeb}
We denote by $X$ the complement of these embedded copies of $E$ and we denote by $Y \subset \bS^2 \sharp \cR(T^2)$ the pre-image of $X \subset \bS^2 \sharp \cR(\bS^2)$ under \eqref{eq:branched-covering}. Notice that:
\begin{align*}
Y &\cong (\bS^2 \smallsetminus V) \sharp \cR(T^2 \smallsetminus (V \sqcup V)) \\
X &\cong (\bS^2 \smallsetminus E) \sharp \cR(\bS^2 \smallsetminus (E \sqcup E)) \cong \cR(\bS^2 \smallsetminus E) \cong S,
\end{align*}
where $V$ denotes the wedge sum of two copies of $E$ at the basepoint $x$. Since we have in particular removed all branch points of the branched double covering, we obtain by restriction a (genuine) double covering
\begin{equation}
\label{eq:double-covering}
Y \longrightarrow X
\end{equation}
depicted in Figure \ref{fig:double-covering}.

We fix compatible basepoints on $X$ and $Y$ and denote by $H$ the index-$2$ subgroup of $\pi_1(X)$ corresponding to this double covering. We also write $\Map_*(X)$ and $\Map_*(Y)$ for the \emph{based} mapping class groups of $X$ and $Y$, given by isotopy classes of self-homeomorphisms that fix the basepoint.

\begin{lem}
\label{lem:preserving-subgroup}
The action of $\Homeo_*(X)$ on $\pi_1(X)$ preserves the subgroup $H$.
\end{lem}

\begin{proof}
We first describe the subgroup $H \subset \pi_1(X)$ intrinsically. A based loop $\gamma$ in $X$ lies in $H$ if and only if its lift to $Y$ is a closed loop. This occurs if and only if the sum of its winding numbers around all branch points of the branched double covering \eqref{eq:branched-covering} is even. We therefore have to show that \emph{if} the sum of these winding numbers is even for $\gamma$, then the same is true for $\varphi \circ \gamma$, where $\varphi$ is any based self-homeomorphism of $X$.

A subtle point here is the meaning of \emph{winding number} (which we only need to define mod $2$): a simple loop in the surface $X$ has winding number $\pm 1$ around an end $e \neq \infty$ if it separates $X$ into two pieces, one containing $e$ and the other containing the end $\infty$. Here $\infty$ denotes the end corresponding to going off to infinity to the right in Figure \ref{fig:double-covering}. More precisely, recall that the end space of $X$ is the one-point compactification $\Upsilon^+(E) = E\omega + 1$ of a countably infinite disjoint union of copies of $E$ and $\infty$ denotes the point at infinity of this one-point compactification. By Corollary~\ref{cor:top-distinguished-compactification} and our assumption that $E$ has a topologically distinguished point, the point $\infty \in E\omega + 1$ is also topologically distinguished. Thus any self-homeomorphism $\varphi$ of $X$ fixes $\infty$, meaning that the notion of ``winding number'' is preserved by $\varphi$.

Let us now show that if the sum of the winding numbers of $\gamma$ around all branch points of $X$ is even, then the same is true for $\varphi \circ \gamma$. The end space $E\omega + 1$ of $X$ has a topologically distinguished subset $\{ x \}\omega$ given by the copy of the topologically distinguished point $x$ in each copy of $E$. But this is precisely the set of branch point of the branched double covering \eqref{eq:branched-covering}. Thus the self-homeomorphism $\varphi$ must send each end of $X$ corresponding to a branch point to another end of $X$ corresponding to a branch point. Its effect on winding numbers around branch points is therefore simply to permute them; so in particular their \emph{sum} is preserved. Hence if the sum of winding numbers around branch points is even for $\gamma$, then the sum of winding numbers around branch points will also be even for $\varphi \circ \gamma$.
\end{proof}

\begin{rem}
The proof of Lemma \ref{lem:preserving-subgroup} is where our assumption that the space $E$ has a topologically distinguished point is used decisively. The lemma would be false without this assumption. See also Remark \ref{rmk:counterexample}.
\end{rem}

We may now complete the proof of Theorem \ref{athm:uncountable} under the stronger assumptions that we are currently making (we explain how to remove these assumptions at the end of this section).

\begin{proof}[Proof of Theorem \ref{athm:uncountable} under additional assumptions.]
It follows from Lemma \ref{lem:preserving-subgroup} that each based homeomorphism of $X$ lifts uniquely to a based homeomorphism of $Y$, giving us a continuous map $\Homeo_*(X) \to \Homeo_*(Y)$, which on $\pi_0$ induces
\begin{equation}
\label{eq:lifting-map}
\Map_*(X) \longrightarrow \Map_*(Y).
\end{equation}
Filling in all \emph{planar} ends of a surface is a functorial operation on the category of surfaces, so by filling in all planar ends of $Y$ we obtain a continuous map $\Homeo_*(Y) \to \Homeo_*(L)$ (see Proposition \ref{prop:extension-continuous}), which on $\pi_0$ induces
\begin{equation}
\label{eq:filling-in-map}
\Map_*(Y) \longrightarrow \Map_*(L).
\end{equation}
Composing \eqref{eq:lifting-map} and \eqref{eq:filling-in-map} with the forgetful map $\Map_*(L) \to \Map(L)$, we obtain
\begin{equation}
\label{eq:from-X-to-L}
\Map_*(X) \longrightarrow \Map(L).
\end{equation}
Let $\alpha_1,\alpha_2,\ldots$ be the collection of simple closed curves on $X$ depicted in Figure \ref{fig:double-covering}. Since $\gamma_i$ is a double covering of $\alpha_i$, we see that
\[
(T_{\alpha_i})^2 \longmapsto T_{\gamma_i}
\]
under \eqref{eq:from-X-to-L}. Now recall that in \S\ref{s:Loch-Ness} (see diagram \eqref{eq:induced-square-modified}) we factored the inclusion $\bigoplus_\fc \bZ \subset \bigoplus_\fc \bQ$ through a map $\bigoplus_\fc \bZ \to \Map(L)$ that sends the generator $1 \in \bZ$ of each summand to a certain infinite product of Dehn twists around the curves $\gamma_i$ from the top of Figure \ref{fig:double-covering}. Replacing each $T_{\gamma_i}$ with $(T_{\alpha_i})^2$ in this infinite product, we obtain a map $\bigoplus_\fc \bZ \to \Map_*(X)$ making the following triangle commute:
\begin{equation}
\label{eq:commutative-triangle}
\begin{tikzcd}
& \displaystyle\bigoplus_{\fc} \bZ \ar[dl] \ar[dr] & \\
\Map_*(X) \ar[rr,"{\eqref{eq:from-X-to-L}}"] && \Map(L),
\end{tikzcd}
\end{equation}
where the right-hand diagonal map is part of a factorisation $\bigoplus_\fc \bZ \to \Map(L) \to \bigoplus_\fc \bQ$ of the standard inclusion. We have therefore shown that the standard inclusion of $\bigoplus_\fc \bZ$ into $\bigoplus_\fc \bQ$ also factors through $\Map_*(X)$. Now consider the diagram
\begin{equation}
\label{eq:factoring-through-Map-star}
\begin{tikzcd}
\bigoplus_\fc \bZ \ar[r] \ar[dr] & \Map_*(X) \ar[r,"\varphi"] \ar[d,two heads] & \bigoplus_\fc \bQ \\
& \Map(X), &
\end{tikzcd}
\end{equation}
where the middle vertical map forgets the basepoint. This is part of the Birman exact sequence for $X$, and its kernel is $\pi_1(X)$, which is in particular countable. Let us denote this kernel by $K$ and consider its image $\varphi(K) \subset \bigoplus_\fc \bQ$. Since $\varphi(K)$ is countable and each of its elements has only finitely many non-zero coordinates in $\bigoplus_\fc \bQ$ (because it is a direct \emph{sum}), it is contained in the subgroup of $\bigoplus_\fc \bQ$ given by the direct sum of only countably many of the copies of $\bQ$. If we take the quotient by this subgroup, the resulting group is again isomorphic to $\bigoplus_\fc \bQ$ and the homomorphism $\varphi$ now descends to $\Map(X)$. On the left-hand side of \eqref{eq:factoring-through-Map-star}, we may compose with the inclusion of the corresponding sub-direct-summand of $\bigoplus_\fc \bZ$ (which is again isomorphic to $\bigoplus_\fc \bZ$); this ensures that the composition across the top row of the following diagram is still the standard inclusion of $\bigoplus_\fc \bZ$ into $\bigoplus_\fc \bQ$:
\begin{equation}
\label{eq:factoring-through-Map}
\begin{tikzcd}
\bigoplus_\fc \bZ \ar[r] \ar[rrrr,bend left=15] & \bigoplus_\fc \bZ \ar[r] \ar[dr] & \Map_*(X) \ar[r,"\varphi"] \ar[d,two heads] & \bigoplus_\fc \bQ \ar[r] & \bigoplus_\fc \bQ \\
&& \Map(X) \ar[urr] &&
\end{tikzcd}
\end{equation}

Thus we have shown that the standard inclusion of $\bigoplus_\fc \bZ$ into $\bigoplus_\fc \bQ$ factors through $\Map(X)$. This standard inclusion induces an injection on homology in all degrees, by Lemma \ref{lem:exterior-Z-to-Q} and the fact that $H_*(A) = \Lambda^*(A)$ for torsion-free abelian groups $A$, so it follows that we have an injection
\[
\Lambda^*\Bigl(\bigoplus_\fc \bZ\Bigr) = H_*\Bigl(\bigoplus_\fc \bZ\Bigr) \longhookrightarrow H_*(\Map(X)) = H_*(\Map(S)).
\]
This completes the proof of Theorem \ref{athm:uncountable} under the additional assumptions on the surface $S$.
\end{proof}

We finish this section by showing how to modify the argument above to allow the more general surfaces $S$ considered in the theorem.

\begin{figure}
    \centering
    \includegraphics[scale=0.8]{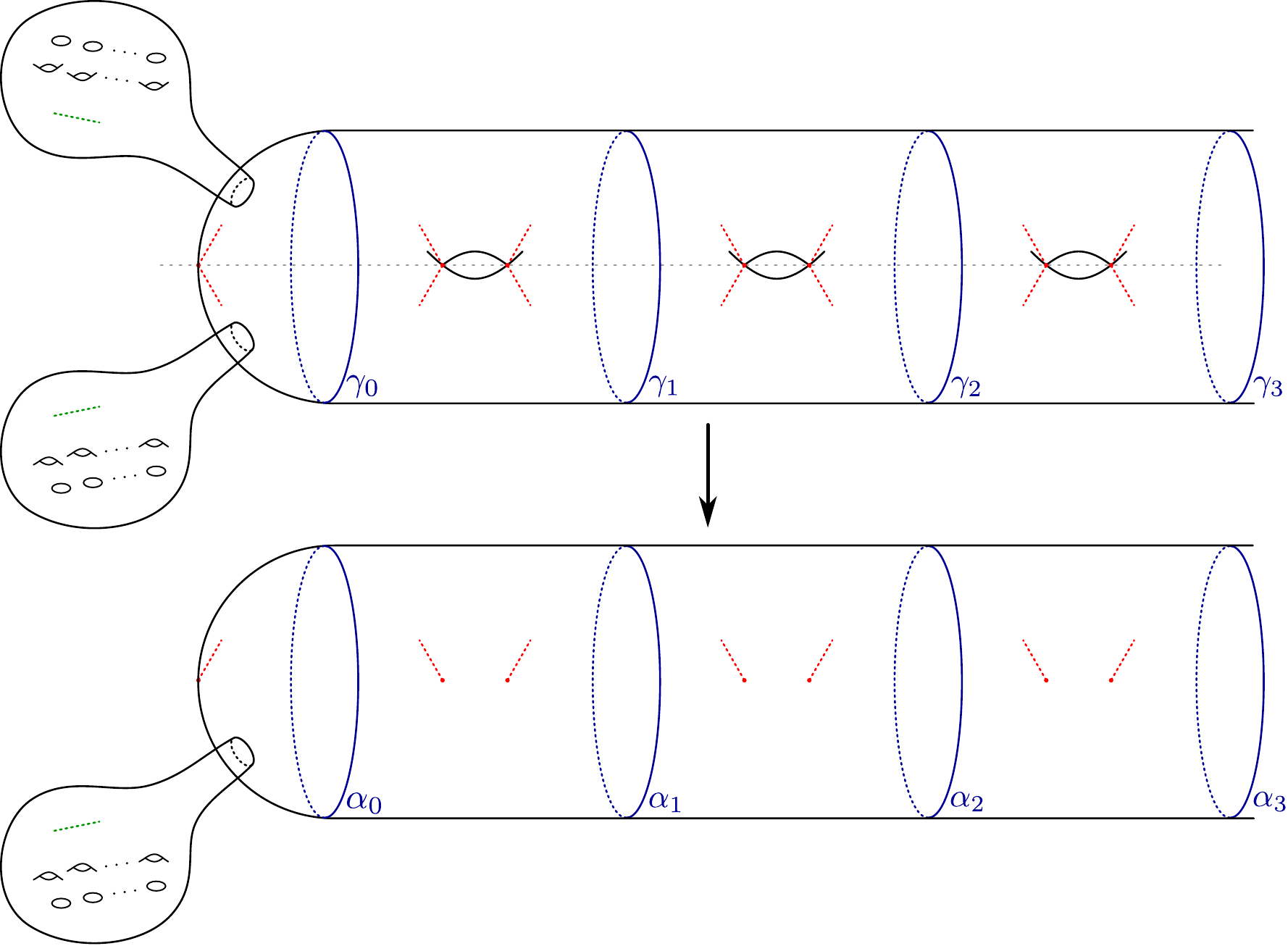}
    \caption{A modification of the branched double covering depicted in Figure \ref{fig:double-covering}.}
    \label{fig:double-covering-extended}
\end{figure}

\begin{proof}[Proof of Theorem \ref{athm:uncountable} in general.]
The proof follows exactly the same strategy as the proof in the special case above, so we just explain the steps that differ slightly.

In general, the surface $S$ is of the form pictured at the bottom of Figure \ref{fig:double-covering-extended}, where we have taken a connected sum of the surface considered previously with another surface of finite genus having finitely many boundary components, such that none of the points of its end space are similar to the topologically distinguished point $x \in E$. We may correspondingly modify the total space of the double covering by taking two connected sums with this surface (no new branch points are introduced).

Lemma \ref{lem:preserving-subgroup} generalises directly to this setting, giving us a homomorphism that lifts (based) mapping classes up the double covering. Filling in all planar ends upstairs, as well as the finitely many boundary components, we obtain (as before) the Loch Ness monster surface $L$. With these modifications, the rest of the proof is identical to the proof in the special case given above, using the constructions of \S\ref{s:Loch-Ness}.
\end{proof}

\begin{rem}
\label{rmk:counterexample}
It is essential to assume in Theorem \ref{athm:uncountable} that $E_2$ has a topologically distinguished point. Indeed, if we do not assume this, then the theorem is false. For example, without this assumption, the theorem would assert that the homology of $\Map(\bS^2 \smallsetminus \cC)$ is uncountable in all positive degrees, since $\Upsilon^+(\cC) \cong \cC$. However, the first and second homology groups of $\Map(\bS^2 \smallsetminus \cC)$ are known to be $0$ and $\bZ/2$ respectively~\cite{CC21}.
\end{rem}

\section{Torsion elements}
\label{s:torsion}

We prove in this section that, whenever $S$ has genus $2$, both $H_1(\PMap(S))$ and $H_1(\Map(S))$ contain an element of order $10$ that generates a direct summand. We first recall that, for \emph{compact} surfaces of genus $2$, the first homology of their mapping class groups is precisely $\bZ/10$. Denote by $S_{g,b}$ the connected, compact, orientable surface of genus $g$ with $b\geq 1$ boundary components. When $g=2$, we have the following.

\begin{thm}[{\cite[\S 5]{Korkmaz02}; see also \cite{Mumford1967}}]
\label{thm-fin-torhmlg}
For any $b\geq 0$, we have $H_1(\Map(S_{2,b})) \cong \bZ/10$, generated by $[T_\alpha]$, where $\alpha$ is any non-separating simple closed curve in $S_{2,b}$.
\end{thm}

\begin{proof}[Proof of Theorem \ref{thm-tor10-hmlg}]
If $S$ has genus $2$, there is an embedding $S_{2,1} \subseteq S$. Also, filling in the ends of $S$ (all of which are planar since it has finite genus) to construct its Freudenthal compactification results in the compact surface $S_{2,b}$, where $b\geq 0$ is the number of boundary components of $S$. We therefore have homomorphisms
\begin{equation}
\label{eq:composition-genus-2}
\Map(S_{2,1}) \longrightarrow \PMap(S) \subseteq \Map(S) \longrightarrow \Map(S_{2,b}),
\end{equation}
where the first is given by extending homeomorphisms of $S_{2,1}$ by the identity on $S \smallsetminus S_{2,1}$ and the second is given by the unique extension of homeomorphisms to the Freudenthal compactification (see Proposition \ref{prop:extension-continuous}). Let $\alpha$ be a non-separating simple closed curve in $S_{2,1}$. By Theorem \ref{thm-fin-torhmlg}, the composition across \eqref{eq:composition-genus-2} induces a map $\bZ/10 \to \bZ/10$ on first homology. Moreover, it clearly sends $[T_\alpha]$ to itself, so it sends a generator of the first $\bZ/10$ to a generator of the second $\bZ/10$; thus it is an isomorphism. Since we have factored an isomorphism of $\bZ/10$ through $H_1(\PMap(S))$ and $H_1(\Map(S))$, it follows that these groups both contain $\bZ/10$ as a direct summand.
\end{proof}

We record here a related general fact that (for example) allows one to embed torsion elements of mapping class groups of compact surfaces into mapping class groups of surfaces of infinite type.

\begin{lem}
\label{lem:embedding}
Let $S$ and $S'$ be two surfaces with non-empty (compact) boundary and assume that $S'$ is planar. Then there are embeddings of direct summands
\begin{equation}
\label{eq:embeddings-of-direct-summands}
H_*(\PMap(S)) \longhookrightarrow H_*(\PMap(S \natural S')) \qquad\text{and}\qquad H_*(\Map(S)) \longhookrightarrow H_*(\Map(S \natural S')),
\end{equation}
where $- \natural -$ denotes the boundary connected sum along one chosen interval in the boundary of each of the two surfaces.
\end{lem}
\begin{proof}
Since $S'$ is planar, it must be of the form $S' = S_{0,b} \smallsetminus E$, where $b\geq 1$ is the number of its boundary components and $E$ is its space of ends. We therefore have homomorphisms
\begin{equation}
\label{eq:composition-boundary-connected-sum}
\PMap(S) \longrightarrow \PMap(S \natural S') \longrightarrow \PMap(S \natural S_{0,b}) \longrightarrow \PMap(S \natural S_{0,1}) \cong \PMap(S),
\end{equation}
where the first is given by extending homeomorphisms by the identity on $S'$, the second is given by the unique extension of homeomorphisms to the Freudenthal compactification (see Proposition \ref{prop:extension-continuous}) and the third is given by filling in all boundary components of $S_{0,b}$ with discs, except the one along which we have taken the boundary connected sum, and extending homeomorphisms by the identity on these new discs. The isomorphism on the right-hand side is induced by a homeomorphism $S \natural S_{0,1} \cong S$ given by pushing the disc $S_{0,1}$ into a collar neighbourhood of the boundary of $S$. The composition across \eqref{eq:composition-boundary-connected-sum} is given by extending homeomorphisms of $S$ by the identity on $S_{0,1}$ and then conjugating by the homeomorphism $S \natural S_{0,1} \cong S$. This is clearly isotopic to the identity, so, applying $H_*$, we have factored the identity map of $H_*(\PMap(S))$ through $H_*(\PMap(S \natural S'))$, which provides the first embedding of \eqref{eq:embeddings-of-direct-summands}. The second embedding follows by an identical argument, replacing $\PMap(-)$ with $\Map(-)$ everywhere.
\end{proof}

\section{Some open problems}
\label{s:open-problems}

In this section we propose some open questions, in addition to Questions \ref{ques-countable}, \ref{ques-torsion-infinite-support} and \ref{ques-torsion-uncountable} discussed in the introduction. We divide them into \S\ref{ss:open-homology} on homology and \S\ref{ss:open-cohomology} on cohomology.

\subsection{Homology.}
\label{ss:open-homology}

So far, our calculations suggest the answer to the following question could be positive.

\begin{ques}
Let $S$ be an infinite-type surface. Suppose that, for some $i\geq 1$, the group $H_i(\Map(S))$ is countable. Is $H_i(\Map(S))$ finitely generated for all $i$?
\end{ques}

This would imply a dichotomy between those $S$ for which $H_i(\Map(S))$ is finitely generated for all $i\geq 1$ and those $S$ for which $H_i(\Map(S))$ is uncountable for all $i\geq 1$.

\begin{ques}
Let $S_{g,1}$ be the connected, compact, orientable surface of genus $g$ and with one boundary component. Does the forgetful map $\Map(S_{g,1} \smallsetminus \CC) \to \Map(S_{g,1})$ induce isomorphisms on homology in all degrees?
\end{ques}

\begin{rem}
When $g=0$, a positive answer follows from \cite[Theorem~B]{PalmerWu1}. The answer in degree one (and for any $g$) has been proven to be positive in \cite[Theorem~2.3]{CalegariChen2022}. On the other hand, the answer would be negative if we considered the sphere instead of $S_{g,1}$, since $H_2(\Map(\bS^2 \smallsetminus \cC)) \cong \bZ/2$ by \cite[Theorem A.2]{CC21}. It would also be negative if we took the plane instead of $S_{g,1}$, since $H_i(\Map(\bR^2 \smallsetminus \cC)) \cong \bZ$ for all even $i$ by \cite[Theorem~A]{PalmerWu1}.
\end{rem}

By \cite[Theorem~C]{PalmerWu1}, the mapping class groups of $1$-holed binary tree surfaces are acyclic. One may wonder whether these are the \emph{only} acyclic mapping class groups of infinite-type surfaces with connected boundary:

\begin{ques}
Let $S$ be an infinite-type surface with a single boundary component and suppose that its mapping class group $\Map(S)$ is acyclic. Is $S$ necessarily  a $1$-holed binary tree surface?
\end{ques}

\subsection{Cohomology.}
\label{ss:open-cohomology}

Most of the results of this paper may be summarised as follows. For any infinite-type surface $S$, the natural inclusion $\bigoplus_\fc \bZ \subset \bigoplus_\fc \bQ$ factors as
\begin{equation}
\label{eq:factorisation-through-PMap}
\bigoplus_\fc \bZ \longrightarrow \mathcal{T}(S) \subseteq \PMapcbar{S} \longrightarrow \bigoplus_\fc \bQ
\end{equation}
and similarly for the \emph{full} mapping class group $\Map(S)$ if $S$ satisfies the conditions of Theorem \ref{athm:uncountable} or if it is the Loch Ness monster surface (Proposition \ref{prop:uncountability-Loch-Ness}). Our results about integral homology then follow from the fact that the natural inclusion $\bigoplus_\fc \bZ \subset \bigoplus_\fc \bQ$ induces an injective homomorphism of exterior algebras $\Lambda^*(\bigoplus_\fc \bZ) \subset \Lambda^*(\bigoplus_\fc \bQ)$ on homology (Lemma \ref{lem:exterior-Z-to-Q}). It is therefore natural to consider also the effect of the factorisation \eqref{eq:factorisation-through-PMap} on integral \emph{cohomology}. However, this factorisation does not tell us anything about cohomology, since the composition across \eqref{eq:factorisation-through-PMap} induces the zero map on cohomology:

\begin{lem}
For each $i\geq 1$, we have:
\begin{align*}
H^i\Bigl( \bigoplus_\fc \bZ \Bigr) &\cong \prod_\fc \bZ \\
H^i\Bigl( \bigoplus_\fc \bQ \Bigr) &\cong \begin{cases}
0 & \text{if } i=1 \\
\bigoplus_{2^\fc} \bQ & \text{if } i\geq 2.
\end{cases}
\end{align*}
In particular, the inclusion $\bigoplus_\fc \bZ \subset \bigoplus_\fc \bQ$ induces the zero map on $H^i$.
\end{lem}

\begin{proof}
The last statement follows from the two calculations, since the induced map on $H^i$ has a rational vector space as its domain, which is a divisible group. Its image must therefore also be divisible, but the only divisible subgroup of $\prod_\fc \bZ$ is the trivial group.

It therefore remains to check the two calculations. The first one follows from the fact that $H_i(\bigoplus_\fc \bZ) \cong \bigoplus_\fc \bZ$ for all $i\geq 1$, the universal coefficient theorem, the fact that $\mathrm{Hom}_\bZ(-,-)$ and $\mathrm{Ext}_\bZ(-,-)$ take direct sums to products in the first variable and $\mathrm{Hom}_\bZ(\bZ,\bZ) \cong \bZ$ and $\mathrm{Ext}_\bZ(\bZ,\bZ) = 0$.

For the second calculation, we again use the universal coefficient theorem, where this time we use the facts that $\mathrm{Hom}_\bZ(\bQ,\bZ) = 0$ and that $\mathrm{Ext}_\bZ(\bQ,\bZ)$ is a rational vector space of dimension $\fc$ (see for example \cite{Wiegold1969}). Thus for $i\geq 2$ we have $H^i(\bigoplus_\fc \bQ) \cong \prod_\fc (\bigoplus_\fc \bQ)$, which is a divisible and torsion-free abelian group, hence a rational vector space, of cardinality (and hence also dimension) $\fc^\fc = 2^\fc$.
\end{proof}

Since the composition across \eqref{eq:factorisation-through-PMap} induces the zero map on cohomology, we cannot deduce anything about $H^*(\PMapcbar{S})$ from this. However, we wonder whether the right-hand map of \eqref{eq:factorisation-through-PMap} is nevertheless injective on cohomology. If it is, it would positively answer the first part of the following question.

\begin{ques}
\label{ques:cohomology}
Let $S$ be an infinite-type surface and $i\geq 2$. Do the groups $H^i(\PMapcbar{S})$ or $H^i(\PMap(S))$ contain a rational vector space of dimension $2^\fc$?
\end{ques}

The second part of this question is motivated by the observation that, in the case when $S$ has infinitely many non-planar ends, the answer is yes. In fact, we have:

\begin{prop}
\label{prop:cohomology-PMap}
Let $S$ be a surface with infinitely many non-planar ends and let $i\geq 2$. Then there is an embedding
\[
\bigoplus_{2^\fc} \bQ \oplus \bigoplus_{2^\fc} \bQ/\bZ \longhookrightarrow H^i(\PMap(S)).
\]
\end{prop}

\begin{proof}
By \cite[Corollary 6]{APV20}, $\PMap(S)$ admits a split-surjection onto the Baer-Specker group $\bZ^\bN$, so $H^i(\bZ^\bN)$ is a summand of $H^i(\PMap(S))$. By the universal coefficient theorem, $H^i(\bZ^\bN)$ has a direct summand of the form $\mathrm{Ext}_\bZ(H_{i-1}(\bZ^\bN),\bZ)$ and we know by Lemma \ref{lem:Baer-Specker} that $H_{i-1}(\bZ^\bN)$ contains a direct summand isomorphic to $\bZ^\bN$. Putting this together, it follows that $H^i(\PMap(S))$ has a direct summand isomorphic to $\mathrm{Ext}_\bZ(\bZ^\bN,\bZ)$. This group is isomorphic to $\bigoplus_{2^\fc} \bQ \oplus \bigoplus_{2^\fc} \bQ/\bZ$, by \cite[Theorem~5]{Nunke1961} (see also \cite[Exercise~2 of \S 99]{Fuchs1973}).
\end{proof}

\renewcommand{\theHsection}{appendixsection.\thesection}
\renewcommand{\thesection}{\Alph{section}}
\setcounter{section}{0}

\section{Abelian groups}
\label{appendix:abelian-groups}

We collect here a few facts about abelian groups that are needed in our proofs. For a comprehensive treatment of the theory of abelian groups, we refer to \cite{Fuchs1970}.

Recall that an abelian group $A$ is called \emph{divisible} if for each element $a \in A$ and positive integer $n$, there is another element $b \in A$ such that $a = nb$. An abelian group $A$ is called \emph{injective} if for every injective homomorphism of abelian groups $\iota \colon B \to C$ and homomorphism $f \colon B \to A$, there is a homomorphism $g \colon C \to A$ such that $g \circ \iota = f$.
By \cite[Theorems~21.1 and 24.5]{Fuchs1970}, an abelian group is divisible if and only if it is injective. In particular:

\begin{lem}
\label{lem:divisible-injective}
Every injective homomorphism from a divisible abelian group to another abelian group admits a retraction.
\end{lem}
\begin{proof}
Let $A$ be a divisible abelian group and let $\iota \colon A \to C$ be an injective homomorphism. Since $A$ is injective, taking $B=A$ and $f=\mathrm{id}$ above, we obtain a retraction of $\iota$.
\end{proof}

\begin{lem}
\label{lem:countable-kernel-injective}
Suppose that we have homomorphisms of abelian groups
\[
\begin{tikzcd}
\displaystyle\bigoplus_{\fc} \bQ \ar[r,"f"] & A \ar[r,"g"] & B
\end{tikzcd}
\]
where $f$ is injective and $g$ has countable kernel. Then, after restricting the direct sum on the left to a subcollection of the same cardinality, the composition $g \circ f$ is also injective.
\end{lem}
\begin{proof}
Consider the subgroup $K := \mathrm{ker}(g \circ f) = f^{-1}(\mathrm{ker}(g)) \subset \bigoplus_\fc \bQ$. Since $\mathrm{ker}(g)$ is countable and $f$ is injective, $K$ is a countable subgroup of $\bigoplus_\fc \bQ$. Each element of $K$ has only finitely many non-zero coordinates in the direct sum and $K$ has countably many elements; thus $K$ is contained in the sub-direct-sum given by countably many $\bQ$ summands. After removing these summands from the direct sum, the composition $g \circ f$ is injective.
\end{proof}

\begin{lem}
\label{lem:exterior-Z-to-Q}
For any set $I$, the canonical inclusion $\bigoplus_I \bZ \hookrightarrow \bigoplus_I \bQ$ induces an injective map of graded abelian groups
\begin{equation}
\label{eq:exterior-Z-to-Q}
\Lambda^*\Bigl(\bigoplus_I \bZ\Bigr) \longhookrightarrow \Lambda^*\Bigl(\bigoplus_I \bQ\Bigr).
\end{equation}
\end{lem}

To prove this, we first recall the following basic calculation:

\begin{lem}
\label{lem:exterior-algebra-of-Z-and-Q}
$\Lambda^*(\bZ) \cong \bZ[0] \oplus \bZ[1]$ and $\Lambda^*(\bQ) \cong \bQ[0] \oplus \bQ[1]$.
\end{lem}
\begin{proof}
The only non-obvious statement is that $\Lambda^i(\bQ) = 0$ for $i\geq 2$. To see this, first recall that
\begin{equation}
\label{eq:Q-tensor-Q}
\bQ \otimes_\bZ \bQ \otimes_\bZ \cdots \otimes_\bZ \bQ \cong \bQ
\end{equation}
via an isomorphism that sends $a_1 \otimes a_2 \otimes \cdots \otimes a_i \mapsto a_1 a_2 \cdots a_i$. The $\bZ$-module $\Lambda^i(\bQ)$ is the quotient of this tensor power by the sub-$\bZ$-module generated by all elements $a_1 \otimes a_2 \otimes \cdots \otimes a_i$ with $a_j = a_k$ for some $j\neq k$. Thus to prove that $\Lambda^i(\bQ) = 0$ we have to show that every rational number is a $\bZ$-linear combination of rational numbers of the form $b^2 a_3 \cdots a_i$. For $i\geq 3$ this is obvious, as we may take $b=1$. For $i=2$, consider a rational number $\frac{p}{q}$, where $p,q \in \bZ$ with $q\neq 0$. Lagrange's four-square theorem implies that we have $pq = a^2 + b^2 + c^2 + d^2$ for integers $a,b,c,d$. Dividing by $q^2$, we deduce that $\frac{p}{q}$ is a sum of four rational squares.
\end{proof}

\begin{proof}[Proof of Lemma \ref{lem:exterior-Z-to-Q}]
By \cite[\S V.6.2, V.6.3]{Brown1982}, for any abelian group $A$ we have
\begin{equation}
\label{eq:decomposition-as-colimit}
\Lambda^*\Bigl(\bigoplus_I A \Bigr) \cong \Lambda^*\Bigl(\underset{J \subseteq I}{\mathrm{colim}} \bigoplus_J A \Bigr) \cong \underset{J \subseteq I}{\mathrm{colim}}\; \Lambda^*\Bigl(\bigoplus_J A \Bigr) \cong \underset{J \subseteq I}{\mathrm{colim}} \bigotimes_J \Lambda^*(A),
\end{equation}
where the colimit is taken over finite subsets $J$ of $I$. For any finite set $J$, the canonical map
\[
\bigotimes_J \Lambda^*(\bZ) \longrightarrow \bigotimes_J \Lambda^*(\bQ)
\]
is injective by Lemma \ref{lem:exterior-algebra-of-Z-and-Q} and the natural isomorphisms \eqref{eq:Q-tensor-Q}. Thus \eqref{eq:exterior-Z-to-Q} is also injective since the colimit on the right-hand side of \eqref{eq:decomposition-as-colimit}, for $A=\bZ$ or $A=\bQ$, is taken over a direct system in which all maps are inclusions of direct summands.
\end{proof}

\section{Extending homeomorphisms to Freudenthal compactifications}
\label{appendix:Freudenthal-compactifications}

\begin{notation}
For a surface $S$, recall that we denote by $\overline{S}$ its \emph{Freudenthal compactification} (see \S\ref{s:inf-type}). We will write $\cP(S) = \Ends(S) \smallsetminus \Ends_{np}(S)$ for its space of \emph{planar} ends. We will also write $\widehat{S} \subseteq \overline{S}$ for the subspace $\overline{S} \smallsetminus \Ends_{np}(S)$ where we have removed all non-planar ends from $\overline{S}$. Equivalently, it is the subspace of $\overline{S}$ consisting of all of its locally Euclidean points: in other words it is the maximal subspace that is a surface. Intuitively, $\widehat{S}$ is the result of ``filling in'' all planar ends $\cP(S)$ of $S$.
\end{notation}

Since every homeomorphism of $S$ extends (necessarily uniquely) to $\overline{S}$ and every homeomorphism of $\overline{S}$ sends $\widehat{S}$ onto itself, we have well-defined injective functions:
\begin{equation}
\label{eq:extending-homeomorphisms}
\Homeo(S) \longrightarrow \Homeo(\overline{S}) \longrightarrow \Homeo(\widehat{S}).
\end{equation}

\begin{prop}
\label{prop:extension-continuous}
With respect to the compact-open topology, the left-hand function in \eqref{eq:extending-homeomorphisms} is a topological embedding and the right-hand function is a homeomorphism.
\end{prop}
\begin{proof}
The fact that the left-hand map is a topological embedding follows from Proposition \ref{prop:extension-continuous-general} below applied to $X=S$. To deal with the right-hand map, first note that $\overline{S}$ is the Freudenthal compactification of $\widehat{S}$ (as well as of $S$), so we have a well-defined function
\begin{equation}
\label{eq:from-hat-to-bar}
\Homeo(\widehat{S}) \longrightarrow \Homeo(\overline{S})
\end{equation}
given by extending homeomorphisms uniquely. It is evidently a set-theoretic inverse for the restriction map $\Homeo(\overline{S}) \to \Homeo(\widehat{S})$; hence both \eqref{eq:from-hat-to-bar} and this restriction map are bijections. Now applying Proposition \ref{prop:extension-continuous-general} to $X = \widehat{S}$, we deduce that \eqref{eq:from-hat-to-bar} is a topological embedding. Since it is also a bijection, this means that it is a homeomorphism, and hence so is its inverse, which is the restriction map on the right-hand side of \eqref{eq:extending-homeomorphisms}.
\end{proof}

\begin{cor}
\label{cor:planar-ends-vs-marked-subspaces}
There is an isomorphism of topological groups $\Homeo(S) \cong \Homeo(\widehat{S},\cP(S))$.
\end{cor}
\begin{proof}
This follows directly from Proposition \ref{prop:extension-continuous}, together with the observation that the image of the composite topological embedding \eqref{eq:extending-homeomorphisms} is precisely $\Homeo(\widehat{S},\cP(S))$, the subspace of $\Homeo(\widehat{S})$ of homeomorphisms sending $\cP(S)$ onto itself.
\end{proof}

\begin{rem}
\label{rmk:planar-ends-vs-marked-subspaces}
Corollary \ref{cor:planar-ends-vs-marked-subspaces} says that filling in the planar ends of a surface and then fixing them setwise does not change anything at the level of homeomorphism groups. This generalises the usual dichotomy between thinking of punctures (isolated planar ends) either as punctures or as marked points.
\end{rem}

\begin{prop}
\label{prop:extension-continuous-general}
Let $X$ be a connected, locally connected, locally compact, Hausdorff and second countable space, write $\overline{X}$ for its Freudenthal compactification and give all homeomorphism groups the compact-open topology. Then the injective function $\Homeo(X) \to \Homeo(\overline{X})$ given by unique extensions of homeomorphsims is a topological embedding, in particular it is continuous.
\end{prop}
\begin{proof}
We begin by rephrasing the statement. The topology on $\Homeo(X)$ induced from the compact-open topology on $\Homeo(\overline{X})$ via the injection $\Homeo(X) \to \Homeo(\overline{X})$ is called the \emph{F-topology}. What we must show is that the F-topology coincides with the compact-open topology. (For the weaker statement that $\Homeo(X) \to \Homeo(\overline{X})$ is continuous, rather than a topological embedding, we would just have to show that the compact-open topology is finer than the F-topology.)

The collection of topologies on $\Homeo(X)$ making both the group operation and the evaluation map $\Homeo(X) \times X \to X$ continuous was studied in \cite{Arens1946}, where it was proven that there exists a \emph{minimum} such topology if $X$ is locally compact and Hausdorff. Moreover, if $X$ is also locally connected, this minimal topology is the compact-open topology. On the other hand, it is proven in \cite{DiConcilio2006} that, if $X$ is rim-compact, Hausdorff and $\overline{X}$ is locally connected at any ideal point, then the F-topology is minimal. Thus, if both sets of hypotheses are satisfied, we may conclude that the F-topology coincides with the compact-open topology, as desired. Indeed, the assumptions of the proposition do imply both sets of hypotheses: in particular rim-compactness is weaker than local compactness (which we have assumed) and our assumptions also imply that the Freudenthal compactification $\overline{X}$ is locally connected at any ideal point; see the paragraph before Theorem 9 in \cite{DiConcilio2013}.
\end{proof}

\bibliographystyle{alpha}
\bibliography{references.bib}

\end{document}